\documentclass[12pt]{article}
\usepackage{color}
\usepackage{latexsym,dsfont}
\usepackage{amssymb, stmaryrd}
\usepackage{amsmath}
\usepackage{amsthm}
\usepackage{enumerate}
\usepackage{hyperref}
\usepackage{graphicx,float}
\usepackage{caption,color}
 
\newtheorem{Theorem}{Theorem}[part]

\newtheorem{Proposition}{Proposition}[part]

\newtheorem{Lemma}{Lemma}[part]

\newtheorem{Remark}{Remark}[part]
\newtheorem{Example}{Example}[part]

\def \L{\mathbf{L}}

\def \R{\mathbb{R}}

\def \E{\mathbb{E}}
\def \F{\mathbb{F}}

\def \P{\mathbb{P}}

\def \Uc{{\cal U}}
\def \Bc{{\cal B}}
\def \Cc{{\cal C}}

\def \Fc{{\cal F}}
\def \Gc{{\cal G}}
\def \Hc{{\cal H}}
\def \Kc{{\cal K}}
\def \Lc{{\cal L}}
\def \Pc{{\cal P}}

\def \Oc{{\cal O}}
 \def \Nc{{\cal N}}

\def \Tc{{\cal T}}
\def \Vc{{\cal V}}
\def \Wc{{\cal W}}

\def \Vc{{\cal V}}

\def \1{{\mathds 1}}

\def \Lb{{\bf L}}

\def \eps{\varepsilon}

\def \ni{\noindent}
\def \ep{\hbox{ }\hfill$\Box$}
\def\reff#1{{\rm(\ref{#1})}}

\newcommand{\nc}{\newcommand}
\nc{\esssup}{\mathop{\mathrm{ess\,sup}}}
\nc{\essinf}{\mathop{\mathrm{ess\,inf}}}

\def\Xb{{\mathbf{X}}}
\def\x{\times}

\def\beqs{\begin{eqnarray*}}
\def\enqs{\end{eqnarray*}}
\def\beq{\begin{eqnarray}}
\def\enq{\end{eqnarray}}

\def\vc{\vartheta}
\def\ur{{\rm u}}
\def\xr{{\rm x}}
\def\Tr{{\rm Tr}}

\def\Ur{{\rm U}}

\def\i{\mbox{\tiny \rm 1}}
\def\Cb{\mathbb C}
\def\S{\mathbf S}
\def\vr{{\rm v}}

\begin{document}

\title{Quenched mass transport of particles\\ towards a target}

\author{Bruno Bouchard\thanks{Universit\'e Paris-Dauphine, PSL University, CNRS, UMR 7534, CEREMADE, 75016 Paris, France, email \texttt{bouchard@ceremade.dauphine.fr}.}~\thanks{This work is partially supported by the ANR project CAESARS (ANR-15-CE05-0024).}
\and
Boualem Djehiche \thanks{Department of Mathematics, KTH Royal Institute of Technology, SE-100 44 Stockholm, Sweden, 
e-mail \texttt{boualem@math.kth.se}.} ~\thanks{Financial support from the Swedish Research Council (VR) Grant no. 2016-04086 is gratefully acknowledged}
\and
Idris Kharroubi \thanks{Sorbonne Universit\'e, Sorbonne Paris Cit\'e, CNRS, Laboratoire de Probabilit\'es Statistique et Mod\'elisation, LPSM, F-75005 Paris, France, email \texttt{idris.kharroubi@upmc.fr}.} 
} 

\maketitle

 \begin{abstract}
We consider the stochastic target problem of finding the   collection of initial {laws} of a mean-field stochastic differential equation {such that}  we  can control its evolution to {ensure that it reaches } a prescribed set of terminal probability distributions, {at a fixed time horizon}.  {Here, laws are considered conditionally to the path of the   Brownian motion that drives the system.}  We establish a version of the geometric dynamic programming principle for the {associated}  reachability sets and prove that the corresponding value function is a viscosity solution of a geometric partial differential equation. {This provides a characterization of the initial masses that can be almost-surely transported towards a given target, along the paths of a stochastic differential equation. Our results} extend   \cite{soner2002dynamic} to our setting. 
\end{abstract}

\bigskip
\noindent\textbf{Mathematics Subject Classification (2010):} 93E20, 60K35, 49L25.

\bigskip
\noindent\textbf{Keywords:} McKean-Vlasov SDEs, dynamic programming, stochastic target, mass transportation, viscosity solutions.


\section{Introduction} 
Stochastic target problems are optimization problems {in which} the controller looks for the value{s} $x$ of {a} state  process $X^{t,x,\nu}$ at time $t$, {so that it can} reach some given  set $K$ at a given terminal time $T$, by choosing an appropriate control $\nu$.  Namely, the objective is to characterize the reachability sets 
\beq\label{reach}
V(t)= \Big\{x\in \R^d:~~X_T^{t,x,\nu}\in K \mbox{ for some admissible control $\nu$}\Big\}
\enq
for $t\in[0,T]$. Such optimization problems were first studied in \cite{ST02a} and \cite{soner2002dynamic} {in which} the function $v(t,x)=1-\mathds{1}_{V(t)}(x)$  is shown to  solve a Hamilton-Jacobi-Bellman   equation, in the viscosity {solution} sense. The main motivation of   \cite{soner2002dynamic,ST02a}  {is}  the so-called super-replication problem, in financial mathematics{: the} controller looks for possible initial endowments such that there exists an investment strategy allowing the terminal wealth to satisfy a {super-hedging} constraint,  {almost-surely} (see e.g. \cite{ElKQ95}). But, the range of applications is obviously much wider. 

Another important type of stochastic target problems concerns the case where the terminal constraint is imposed on the mean value of a function of the controlled process. In this case the reachability sets take the following form{:}
\beq\label{reach-mean}
{V_{\ell}}(t) = \Big\{x\in \R^d:~~\E[\ell(X_T^{t,x,\nu})]\geq 0 \mbox{ for some admissible control $\nu$}\Big\},\;~~\quad 
\enq
for $t\in[0,T]$.  This type of constraints is also {common in} financial applications. Indeed, the super-replication price is usually too high to be accepted by buyers. {This is a motivation for relaxing the a.s.~super-hedging criteria by only asking that  $X^{t,x,\nu}\in K$  holds}{, for instance,}  with a (high) probability $p<1$. In this case, the function $\ell$ takes the form $\ell(x)=\mathds{1}_{K}(x)-p$. {For $p=1$, one retrieves \eqref{reach}.} This approach was  introduced in \cite{FL99} {and further developed in} \cite{BET10} where the authors take advantage of the martingale representation theorem to transform the  constraint given in terms of the mean value into an almost-sure constraint.

One of the motivations of this paper is to study the stochastic target problem \eqref{reach-mean} in the case of a  mean-field  (or McKean-Vlasov) controlled diffusion{:}
\beqs
X^{t,\chi,\nu}_s = {\chi}+\int_t^sb_u(X^{t,\chi,\nu}_u,\P_{X^{t,\chi,\nu}_u},\nu_u)du+\int_t^s\sigma_u(X^{t,\chi,\nu}_u,\P_{X^{t,\chi,\nu}_u},\nu_u)dB_u,
\enqs
where $\P_{X^{t,\chi,\nu}_u}$ is the marginal law of $X^{t,\chi,\nu}_{u}$ under $\P$, $B$ is a standard Brownian motion {and $\chi$ is an independent  random variable whose distribution can be interpreted as the initial probability distribution of a population}. This type of stochastic target problems can be embedded into a more general class of problems involving the conditional laws given the Brownian path. Indeed, using the martingale representation theorem as in \cite{BET10}, the constraint in \reff{reach-mean} can be rewritten as
\beqs\begin{array}{lll}
\E_B[\ell(X_T^{t,\chi,\nu})]-\int_t^T\alpha_sdB_s\geq 0
\;   \mbox{ for some  controls $\nu$ and $\alpha$}\;,~~\quad 
 \end{array}
\enqs
where $\E_B$ denotes the conditional expectation given $B$.
In particular, if  we define the control $\bar \nu=(\nu,\alpha)$ and the controlled process $\bar X^{t, (\chi,0),\bar \nu}= (X^{t, \chi, \nu},\int_t^.\alpha dB)$, {this reads }
\beqs
 L\big(\P^B_{\bar X_T^{t,(\chi,0),\bar \nu}}\big)\geq 0 \mbox{ for some   control $\bar \nu$} ,
\enqs
{in which $\P^B_\zeta$ denotes the conditional law of a random variable $\zeta$ given $B$,}   
and
\beqs
L(\mu) & = & \int_{\R^d\times{\R}}(\ell(x)-y)\mu(dx,dy).
\enqs
These considerations suggest {to study a general constraint:}  
\beqs
 \P^B_{X_T^{t,\chi,\nu}}\in G \mbox{ for some admissible control $\nu$},
\enqs
in which $X^{t,\chi,\nu}$ is now defined by 
\beq\label{quenched-sde}
X^{t,\chi,\nu}_{s} =\chi+\int_{t}^{s} b_u(X^{t,\chi,\nu}_u,\P^B_{X^{t,\chi,\nu}_u},\nu_u)du+\int_{t}^{s} \sigma_u(X^{t,\chi,\nu}_u,\P^B_{X^{t,\chi,\nu}_u},\nu_u)d B_u,~~
\enq
 $G$ is a Borel subset of probability measures and $\chi$ is the (random) initial position.

This general formulation is of importance on its own right as it is related to the probabilistic analysis of large scale particle systems, e.g. polymers in  random media,  {in which} one is interested in the behavior of  particles  conditionally on the environment. {This is} also known as  `quenched' {behaviors/properties} {(quenched law of large numbers, quenched large deviations etc.)}, which is  in general different from the so-called `annealed' behaviors obtained by averaging over the underlying random environment (see e.g.~\cite{birkner10,giacomin07,DM89} and the references therein). 
For diffusion processes, quenching boils down to making the drift and diffusion coefficients dependent on the conditional marginal law given the environment, while annealing corresponds to the case where the coefficients depend on the unconditional marginal law (see e.g.~\cite{DM89}).  We therefore coin the term quenched diffusion instead of conditional diffusion to refer to SDEs of the form \reff{quenched-sde}. For our stochastic target problem, the constraint $\P^B_{X_T}\in G$ imposed on the conditional law of the diffusion process is a quenched property for the underlying process.

 One can also further identify the initial condition $\chi$ as a law $\mu$. Then, our problem can be interpreted as a transport problem. What is the collection of initial distributions $\mu$ of a population of particles, that all have the same dynamics, such that the terminal conditional law $ \P^B_{X_T^{t,\chi,\nu}}$, given the environment modeled by the Brownian path $B$, satisfies a certain constraint? This amounts to asking what kind of masses can be transported along the SDE so as to reach a certain set, almost-surely, at $T$:
\begin{align}\label{eq: probleme cibe intro}
\Vc(t)   =    \Big\{ \mu:~\exists (\chi,\nu)~\mbox{s.t.}~\P^B_{\chi}=\mu~\mbox{ and }~\P^B_{X^{t,\chi,\nu}_T}\in  G \Big\}.
\end{align}

This type of viability problems appears naturally in statistical physics. It is also encountered in e.g.~agricultural crop management, as highlighted in Example \ref{example} below.

The rest of the paper is organized as follows. In Section \ref{sec-descr}, we describe in details the quenched controlled diffusion. {We provide some (expected) existence and stability results, together with a conditioning property.} Section \ref{sectarget-DPP} is devoted to the detailed presentation of the quenched stochastic target problem \reff{eq: probleme cibe intro}.  {We prove that it admits a }geometric dynamic programming principle.  {This is the main result of the paper}.  Then, {one can combine the technologies developped in   \cite{carda12,CCD15}  and  \cite{soner2002dynamic} to  derive in Section \ref{PDEdyn} the associated Hamilton-Jacobi-Bellman equation, which extends the main result of  \cite{soner2002dynamic}  to our context. In Section \ref{sec: Uccirc}, we provide an alternative formulation which is more adapted to the case where the reachability set is a half space in one direction (see \cite{soner2009dynamic}), we also comment on the choice of the class of controls, and provide an interpretation in terms of control of the law of a population of particles. } 

\section{Quenched mean-field SDE}\label{sec-descr}

We first describe our probabilistic setting. The $d$-dimensional Brownian motion is constructed on the canonical space in a usual way. More precisely, given a fixed time horizon $T > 0$,  we let $\Omega^{\circ}$ denote the space of continuous $\R^{d}$-valued functions on $[0,T]$, starting at $0$, and  let $\F^{\circ}=(\Fc^{\circ}_{t})_{t\le T}$ denote the filtration generated by the canonical process $B(\omega^{\circ}):=\omega^{\circ}$, $\omega^{\circ} \in \Omega^{\circ}$. We set $\Fc^{\circ}=\Fc^{\circ}_T$ and endow $(\Omega^{\circ},\Fc^{\circ})$ with the Wiener measure $\P^{\circ}$. Later on, $\bar \F^{\circ}=(\bar \Fc^{\circ}_{t})_{t\le T}$ will denote the $\P^{\circ}$-completion of $\F^{\circ}$. 

In order to model the initial probability distribution  of the population, we let  $\Omega^{\i}:=[0,1]^{d}$ be endowed with its Borel $\sigma$-algebra $\Fc^{\i}:=\Bc([0,1]^d)$ and the Lebegues measure $\P^{\i}$. It supports the  $[0,1]^{d}$-uniformly distributed random variable $\xi(\omega^{\i})=\omega^{\i}$,   $\omega^{\i}\in\Omega^{\i}$.
We then define the product filtered space $(\Omega,\Fc,\F,\P)$ by setting $\Omega:=\Omega^{\circ}\times\Omega^{\i}$, $\P=\P^{\circ}\otimes \P^{\i}$,  $\Fc=\Fc_{T}$ where $\F=(\Fc_{t})_{t\le T}$   is the augmentation of $(\Fc_{t}^{\circ}\otimes\Fc^{\i})_{t\le T}$. {From now on, any identity involving random variables has to be taken in $\P$-a.s.~sense. } 
We canonically extend the random variable $\xi$ and the process $B$ on $\Omega$ by setting $\xi(\omega)=\xi(\omega^{\i})$ and $B(\omega)=B(\omega^{\circ})$ for any $\omega=(\omega^{\circ},\omega^{\i})\in \Omega$. We still denote by $\F^{\circ}$ the filtration generated by the extended process $B$ on $\Omega$. Note that it follows from \cite[Chapter 2, Theorem 6.15 and Proposition 7.7]{KS91} applied to the process $(t,\omega)\in[0,T]\times\Omega\mapsto(\xi(\omega),B_t(\omega))$ that $\F$ is right continuous. 

Given a random variable $Y\in \L_{0}(\Omega,\Fc,\P;\R^{d})$ (resp. $Y\in \L_{1}(\Omega,\Fc,\P;\R^{d})$), we   let   $\P_Y^B$ (resp. $\E_B[Y]$) denote a regular conditional law (resp. expectation) under $\P$ of the random variable $Y$  given $(B_t)_{t\leq T}$  on $\R^{d}$.  
In particular, we have the following identifications
\beq
\P^B_Y(A,\omega) & = & \P^{\i}_{Y(\omega^{\circ},.)}(A)\label{id-cond-law}\\
 \E_B\big[Y  \big](\omega) & = & \E^{\i}\big[Y(\omega^{\circ},.)\big]\;  \label{id-cond-exp}
\enq
 for any $\omega=(\omega^{\circ},\omega^{\i})\in\Omega$ and any $A\in\Bc(\R^d)$. Here, $\E^{\i}$ denotes the expectation under $\P^{\i}$ and $\P^{\i}_{Y(\omega^{\circ},.)}$ denotes the law under $\P^{\i}$ of the random variable defined on $\Omega^{\i}$ by $Y(\omega^{\circ},.)(\omega^{\i})=Y(\omega^{\circ},\omega^{\i})$.
We let  $\Pc({\rm S})$  denote the space of probability measures on a Borel space {$({\rm S},\Bc({\rm S}))$}, and define 
\beqs
\Pc_{2} & := & \left\{\mu \in \Pc(\R^d,\Bc(\R^d)) \text{ s.t. } \int_{\R^d} {|x|}^2\mu(dx)<+\infty\right\},
\enqs
where $|x|$ is the Euclidean norm of $x$. 
This space is endowed with the $2$-Wasserstein distance defined by
\begin{align*}
\Wc_2(\mu,\mu')  := & \Big(\inf\Big\{\int_{\R^d\times\R^d}|x-y|^2\pi(dy,dy):~
  \pi \in \Pc(\R^d\times\R^d, \Bc(\R^d\times\R^d))  \\ 
 & \qquad\qquad\qquad \text{ s.t. } \;\pi(\cdot\times\R^d)= \mu \mbox{ and }\pi(\R^d\times\cdot)=\mu'\Big\}\Big)^{1\over 2}\;,
\end{align*}
for $\mu,\mu'\in \Pc_{2}$. For later use, we also define the collection $\Pc_{2}^{\bar \F^{\circ}}$ of $\bar \F^{\circ}$-adapted continuous $\Pc_{2}$-valued processes. 

Let now $\Ur$ be a closed subset of $\R^q$ for some $q\geq1$ and denote by $\Uc$ the collection of $\Ur$-valued $\F$-{progressively measurable} processes. This will be the set of controls. 
Let $\bar \Tc^{\circ}$ denote the set of $[0,T]$-valued $\bar \F^{\circ}$-stopping times. Given $\theta\in \bar \Tc^{\circ}$ and  $\chi\in \Xb^{2}_{\theta}:=\L^{2}(\Omega,\Fc_{\theta},\P;\R^{d})$, $\nu \in \Uc$, and $(b,a):[0,T]\x \R^{d}\x \Pc_{2}\x \Ur\longrightarrow \R^{d}\times \R^{d\times d}$, we let $X^{\theta,\chi,\nu}$ denote the solution   of 
\begin{equation}
X_{\cdot} = \E[\chi| \Fc_{\theta\wedge \cdot}]+\int_{\theta}^{\theta \vee \cdot} b_{s}\big(X_{s},\P_{X_{s}}^B,\nu_{s}\big)ds  \label{eq: SDEMF}
 + \int_{\theta}^{\theta \vee \cdot} a_{s}\big(X_{s},\P_{X_{s}}^B,\nu_{s}\big)dB_{s},
\end{equation}
in which   $(b,a)$ is assumed to be continuous, bounded and satisfies:

\vspace{2mm}

\noindent  \textbf{(H1)} There exists a constant $L$ such that
\beqs
|b_{t}(x,\mu,\cdot)-b_{t}(x',\mu',\cdot)|+|a_{t}(x,\mu,\cdot)-a_{t}(x',\mu',\cdot)| \leq L\Big(|x-x'|+\Wc_2(\mu,\mu')  \Big)
\enqs
for all $t\in[0,T]$, $x,x'\in\R^d$ and $\mu,\mu'\in\Pc_2$. 

\vspace{2mm}

The term  $\E[\chi| \Fc^{}_{\theta\wedge \cdot}]$ in \reff{eq: SDEMF} allows to define $X$ as a continuous adapted process on $[0,T]$, which is done for convenience of notations. One could obviously only consider the process on $[\![\theta,T]\!]$. 

{\begin{Remark} {\rm Note that the controls can depend on the initial value of $\chi$. One could also restrict to $\bar \F^{\circ}$-{progressively measurable} processes, see Section \ref{sec: Uccirc} for a discussion.} \end{Remark}} 

The above condition ensures as usual that a unique strong solution to \reff{eq: SDEMF} can indeed be defined.  
 
\begin{Proposition}\label{prop: existence unicite eds} For all $\theta\in \bar \Tc^{\circ}$, $\nu\in\Uc$ and  $\chi\in \Xb^{2}_{\theta}$, \reff{eq: SDEMF} admits a unique strong solution $X^{\theta,\chi,\nu}$, and it  satisfies 
\begin{align}\label{eq: X carre int}
\E\Big[\sup_{s\in [0,T]}|X_s^{\theta,\chi,\nu}|^{2}\Big] <  +\infty\;.
\end{align} 
Moreover, for all $(t,\chi,\nu )\in [0,T]\x \Xb_{t}^{2}\x \Uc$, if $t_{n}\to t$, $\chi_{n}\to \chi$ in $\L_{2}$ with  $\chi_{n}\in \Xb_{t_{n}}^{2}$  for all $n$,  and $(\nu^{n})_{n}\subset \Uc$ converges to $\nu$ $dt\x d\P$-a.e., then 
\begin{align}\label{eq: continuity loi par rapport condi initial}
\lim_{n\to \infty }\E[\Wc_2(\P^{B}_{X^{t_{n},\chi_{n},\nu^{n}}_{T}},\P^{B}_{X^{t,\chi,\nu}_{T}})^{2}]=0.  
\end{align}
\end{Proposition}º

\proof 1. The estimate \reff{eq: X carre int} is a consequence of the boundedness of $(b,a)$. \\
\noindent 2. Existence  follows from a  similar fixed point argument as in \cite{JMW10} (see also~\cite{S89} and \cite{ DV95,V88} for the martingale problem approach). Since we work in a slightly different context, we provide the proof for completeness. 
  
  \noindent 2.a.  Let $\Cb$ denote the space of continuous $\R^{d}$-valued maps on $[0,T]$ endowed with the sup-norm topology and $\Pc_2(\Cb,\Bc(\Cb))$ denote the set of probability measures $\hat P$ on $(\Cb,\Bc(\Cb))$ such that $ \int_{\Cb} \sup_{s\le T} |f_{s}|^{2} \;\hat P(df)<\infty$. For $\hat Q, \hat P \in \Pc_2(\Cb,\Bc(\Cb))$ and $t\le T$, we define the Wasserstein metric:
\begin{align*}
D_{t}(\hat P,\hat Q):=&\inf \big\{  \int_{\Cb^2} \sup_{{0\le s \le t}} |f_{s}-g_{s}|^{2} \;\hat R(df, dg)  : 
 \hat R \in \Pc(\Cb^2,\Bc(\Cb^2))\\
 & \qquad\qquad\qquad\qquad \text{ s.t. } \;\hat R(\cdot\times\Cb)= \hat P \mbox{ and }\hat R(\Cb\times\cdot)=\hat Q \big\}^{1\over 2}.
\end{align*}
If $\hat Q\in  \Pc_{2}(\Cb,\Bc(\Cb))$ has time marginals $(\hat Q_{s})_{s\le T}$ then 
$$
\Wc_2(\hat Q_{t},\hat Q_{s})^2\le  \int_{\Cb} |Y_{t}-Y_{s}|^{2} \hat Q(dY)
$$
so that $\Wc_2(\hat Q_{t},\hat Q_{s}) \to 0$ as $s\to t$, by dominated convergence. Hence, $(\hat Q_{s})_{s\le T}$ is continuous.  

\noindent 2.b. Let $\S_{2}$ denote the set of continuous adapted $\R^{d}$-valued processes $Z$ such that $\|Z\|_{\S_{2}}:=\E[\sup_{[0,T]}|Z|^{2}]^{\frac12}<\infty$. Let $\L_{2}(\Omega^{\circ}; \Pc_{2}(\Cb,\Bc(\Cb)))$ be the collection of   random variables defined on $\Omega^{\circ}$ and with values in $\Pc_{2}(\Cb,\Bc(\Cb))$, with finite norm $\E[\|\cdot\|^{2}_{\Pc_{2}(\Cb,\Bc(\Cb))}]^{\frac12}$. Let $\Phi$ be the map that to $\bar Q \in \L_{2}(\Omega^{\circ}; \Pc_{2}(\Cb,$ $\Bc(\Cb)))$ associates $\P_{X^{\bar Q}}^{B}\in \L_{2}(\Omega^{\circ}; \Pc_{2}(\Cb,\Bc(\Cb)))$ in which $\P_{X^{\bar Q}}^{B}(\omega^{\circ})$ is a regular conditional law of $X^{\bar Q}$ given $\omega^{\circ}\in \Omega^{\circ}$ with  $X^{\bar Q}$ defined as the solution of 
\begin{align*}
X^{\bar Q}_{\cdot} =& \E[\chi|\bar \Fc^{\circ}_{\theta\wedge \cdot}] +\int_{\theta}^{\theta \vee \cdot} b_{s}\big(X^{\bar Q}_{s},\bar Q_{s},\nu_{s}\big)ds
 + \int_{\theta}^{\theta \vee \cdot} a_{s}\big(X^{\bar Q}_{s},\bar Q_{s},\nu_{s}\big)dB_{s} ,
\end{align*}
and where  $\bar Q_{s}(\omega^{\circ})$ is the $s$-marginal of $\bar Q(\omega^{\circ})$ for $\omega^{\circ}\in \Omega^{\circ}$. It follows from 2.a.~that $\P_{X^{\bar Q}}^{B}(\omega^{\circ})$ has continuous path, for $\P^{\circ}$-a.e.~$\omega^{\circ}\in \Omega^{\circ}$.  By repeating the arguments in \cite[Proof of Proposition 2]{JMW10}, see also 3.~below, we obtain that $\Phi$ is contracting. Since $\L_{2}(\Omega^{\circ}; \Pc_{2}(\Cb,\Bc(\Cb)))$ is complete, it follows that $\Phi$ admits a fix point $\bar Q$. 

\noindent 3. It remains to prove our last estimate. 
The Lipschitz continuity and boundedness of $(b,a)$ combined with Burkholder-Davis-Gundy inequality implies that one can find $C>0$, that only depends on $(b,a)$, such that 
\begin{align*}
&\E[\sup_{u\in [0,s]} |X_u^{t,\chi,\nu}-X_u^{t_{n},\chi_{n},\nu^{n}}|^{2}]\\
\le& C (|t-t_{n}|+\E[|\chi-\chi_n|^{2}])\\
&+ C \E\left[\int_{0}^{s}\left( \sup_{u\in [0,r]} |X_u^{t,\chi,\nu}-X_u^{t_n,\chi_n,\nu^{n}}|^{2}+\Wc_2^{2}(\P^{B}_{X^{t,\chi,\nu}_{r}},\P^{B}_{X^{t_{n},x_{n},\nu^{n}}_{r}})\right)dr \right] \\
& + C \E\left[\int_{0}^{s}|b_{r}(X^{t,\chi,\nu}_{r},\P^{B}_{X^{t,\chi,\nu}_{r}},\nu_{r})-b_{r}(X^{t,\chi,\nu}_{r},\P^{B}_{X^{t,\chi,\nu}_{r}},\nu^n_{r})|^{2}dr\right]\\
&+ C \E\left[\int_{0}^{s}|a_{r}(X^{t,\chi,\nu}_{r},\P^{B}_{X^{t,\chi,\nu}_{r}},\nu_{r})-a_{r}(X^{t,\chi,\nu}_{r},\P^{B}_{X^{t,\chi,\nu}_{r}},\nu^{n}_{r})|^{2} dr\right].
\end{align*}
Since 
\begin{align*}
\E[\Wc_2^{2}(\P^{B}_{X^{t,\chi,\nu}_{r}},\P^{B}_{X^{t_{n},x_{n},\nu^{n}}_{r}})]&\le \E[D_{r}^{2}(\P^{B}_{X^{t,\chi,\nu}},\P^{B}_{X^{t_{n},x_{n},\nu^{n}}})]\\
&\le  \E[\sup_{u\in [0,r]} |X_u^{t,\chi,\nu}-X_u^{t_{n},x_{n},\nu^{n}}|^{2}],
\end{align*}
by Gronwall's Lemma we obtain (for a different constant $C>0$)
\begin{align*}
&\E[\Wc_2^{2}(\P^{B}_{X^{t,\chi,\nu}_{T}},\P^{B}_{X^{t_{n},x_{n},\nu^{n}}_{T}})]\\
\le &\;\E[\sup_{u\in [0,T]} |X_u^{t,\chi,\nu}-X_u^{t_{n},x_{n},\nu^{n}}|^{2}]\\
\le &\;C (|t-t_{n}|+\E[|\chi-\chi_n|^{2}])\\
& + C \E\left[\int_{0}^{T}|b_{r}(X^{t,\chi,\nu}_{r},\P^{B}_{X^{t,\chi,\nu}_{r}},\nu_{r})-b_{r}(X^{t,\chi,\nu}_{r},\P^{B}_{X^{t,\chi,\nu}_{r}},\nu^{n}_{r})|^{2}dr\right]\\
&+ C \E\left[\int_{0}^{T}|a_{r}(X^{t,\chi,\nu}_{r},\P^{B}_{X^{t,\chi,\nu}_{r}},\nu_{r})-a_{r}(X^{t,\chi,\nu}_{r},\P^{B}_{X^{t,\chi,\nu}_{r}},\nu^{n}_{r})|^{2} dr\right].
\end{align*}
The function $(b,a)$ being continuous and bounded, the required result follows. 
\ep \\

\begin{Remark}{\rm We can construct a particle approximation for the SDE \reff{eq: SDEMF} as follows. We first note that for $t\in[0,T]$, $\chi\in\Xb_t$ and $\nu\in \Uc$ there exist Borel maps $\xr$ and $\ur$ such that  $\chi=\xr((B_{s})_{s\le t}, \xi^1)$ $\P$-a.s.~and $\nu=\ur(\cdot, (B_{s})_{s\le \cdot}, \xi^1)$, up to modification. 
We then consider a sequence $(\xi^{\ell})_{\ell \geq 1}$ of $i.i.d.$ random variable with uniform law on $[0,1]^d$ and  independent of $B$ and we set  $(\chi^{\ell},\nu^{\ell}):=(\xr((B_{s})_{s\le t},\xi^\ell), \ur(\cdot, (B_{s})_{s\le \cdot}, \xi^{\ell}))$, for $\ell\ge 1$.

For $n\geq \ell\geq 1$ we define $X^{\ell}$  and $X^{n,\ell}$ as the  respective solutions to the SDEs:
\beqs
X_{\cdot}^{\ell} & = & \chi^\ell+\int_{t}^{ \cdot} b_{s}\big(X^{\ell}_{s},\P^B_{X^\ell_s},\nu_{s}^\ell\big)ds  
 + \int_{t}^{ \cdot} a_{s}\big(X^{\ell}_{s},\P^B_{X^\ell_s},\nu_{s}^\ell\big)dB_{s},
\enqs 
and
\beqs
X_{\cdot}^{n,\ell} & = & \chi^\ell+\int_{t}^{ \cdot} b_{s}\big(X^{n,\ell}_{s},\bar\mu^n_s,\nu_{s}^\ell\big)ds  
 + \int_{t}^{ \cdot} a_{s}\big(X^{n,\ell}_{s},\bar\mu^n_s,\nu_{s}^\ell\big)dB_{s},
\enqs 
where the measures $\bar \mu^n$, $n\geq 1$ are defined by
\beqs
\bar \mu ^n_s & := & {1\over n}\sum_{\ell=1}^n\delta_{{X^{n,\ell}_s}}\;,\quad s\geq 0\;.
\enqs
Then, following the same arguments as {in} \cite[Theorem 3]{JMW10}, we have 
\beqs
\lim_{n\rightarrow+\infty}\sup_{\ell\leq n}\E^1\Big[\sup_{u\in [0,T]}\big|X_u^{n,\ell}-X_u^\ell\big|^2\Big] & = & 0\;.
\enqs
 In particular, this induces the convergence of empirical  measures $\bar \mu^n$:
 \beqs
\lim_{n\rightarrow+\infty}\Wc_2(\bar \mu^n_s,\P^B_{X^1_s}) & = & 0\;,\quad s\in[0,T]\;.
\enqs}
\end{Remark}
In the sequel, we denote by $^t\omega^\circ$ the element $(\omega^\circ_{s\wedge t})_{s\in[0,T]}$ for $\omega^\circ\in\Omega^\circ$ and $t\in[0,T]$. 
 We note that the solution can also be defined  $\omega^{\i}$ by $\omega^{\i}$. More precisely, we have the following. 

\begin{Proposition}\label{prop: eds def condi omega 1} Fix $\theta\in \bar \Fc^{\circ}$,  $\chi\in \Xb^{2}_{\theta}$ and $\nu \in \Uc$.  Let $X^{Q}$ be the solution of \reff{eq: SDEMF} with $Q=(Q_{s})_{s\le T}\in \Pc_{2}^{\bar \F^{\circ}}$ in place of $(\P_{X_{s}}^{B})_{s\le T}$. Then, there exists  Borel  measurable   maps $\xr:\Omega^{\circ}\x \Omega^{\i}\to \R^{d}$ and  $\ur:[0,T]\x\Omega^{\circ}\x \Omega^{\i}\to \Ur$  such that $\chi=\xr(^{\theta}B,\xi)$ $\P$-a.s. and  $\nu_{\cdot}={\ur}_{\cdot}({^\cdot}B,\xi)$ $dt\x \P$-a.e.~on $[0,T]\times\Omega$, such that, for all stopping time $\tau$,  $X^{Q,\omega^{\i}}_{\tau\vee \theta}=X^{Q}_{\tau\vee \theta}(\cdot,\omega^{\i})$ $\P^{\circ}$-a.s.~for $\P^{\i}$-a.e.~$\omega^{\i}\in \Omega^{\i}$, in which $X^{Q,\omega^{\i}}$ solves 
\begin{align*}
X_{\cdot}^{Q,\omega^{\i}} =& \E[\xr(B,\omega^{\i})|\Fc_{\cdot\wedge \theta}]+\int_{\theta}^{\theta \vee \cdot} b_{s}\big(X^{Q,\omega^{\i}}_{s},Q_{s},\ur_{s}(^sB,\omega^{\i})\big)ds\\
& + \int_{\theta}^{\theta \vee \cdot} a_{s}\big(X^{Q,\omega^{\i}}_{s},Q_{s},\ur_{s}({}^sB,\omega^{\i})\big)dB_{s}.
\end{align*}
Moreover, the map $\omega^{\i}\in \Omega^{\i}\mapsto X^{Q,\omega^{\i}}_{\tau\vee \theta}\in \L_{2}(\Omega^{\i},\Fc^{\i},\P^{\i}; \L_{2}(\Omega^\circ,\Fc^{\circ}_T,\P^{\circ};\R^d))$ is measurable.
\end{Proposition}
\proof The existence of the Borel maps $\xr$ and $\ur$ is standard, and it  is not difficult to prove that $\omega^{\i}\in \Omega^{\i}\mapsto X^{Q,\omega^{\i}}_{\tau}\in  \L_{2}(\Omega^\circ,\Fc^{\circ}_T,\P^{\circ};\R^d)$ is measurable because $a$ and $b$ are continuous and bounded.  Standard estimates then show that $\E[|X^{Q,\xi}_{\tau\vee \theta}-X^{Q}_{\tau\vee \theta}|^{2}|\xi]=0$. 
 \ep\\

For later use, we now show that the law of $(X^{t,\chi,\nu},B)$ actually only depends on the joint law of $(\chi,\nu,^{t}\!B)$.

\begin{Proposition}\label{prop: equiv loi jointe} Let $\xr: \Omega^{\circ}\x \Omega^{\i} \to \R^{d}$ and $\ur:[0,T]\x\Omega^{\circ}\x \Omega^{\i}\to \Ur$ be Borel maps such that $\chi:=\xr(^{t}\!B,\xi)\in {\Xb^{2}_{t}}$ and $\nu:=\ur_{\cdot}(B,\xi)\in \Uc$. 
Let $\bar \xi$ and $\bar \xi'$ be $[0,1]^{d}$-valued $\Fc_{t}$-measurable and set $\bar \chi:=\xr(^{t}\!B,\bar \xi)$ and $\bar \nu:=\ur_{\cdot}(B,\bar \xi')$. Assume that $(\chi,{\nu_{\cdot\vee t}},^{t}\!B)$ and $(\bar \chi,{\bar \nu_{\cdot\vee t}},^{t}\!B)$ have the same law. Then, $(X^{t,\chi,\nu},B)$ and $(X^{t,\bar \chi,\bar \nu},B)$ have the same law.  
\end{Proposition}
 
\proof One can follow \cite[Theorem 3.3]{claisse2013note}.
In their case, the conditioning is made with respect to $^{t}\!B$, in our case it has to be done with respect to $(^{t}\!B,\xi)$, where $\xi$ is independent of $B$, so that the equation can actually be solved conditionally to $\xi$, see Proposition \ref{prop: eds def condi omega 1}. Given the fixed point procedure used in Step 2.~of the proof of Proposition \ref{prop: existence unicite eds} above,   one can then find a sequence $(\hat P^{n})_{n\ge 1}\subset \L_{2}(\Omega^{\circ},\Pc_{2}(\Cb,\Bc(\Cb)))$ such that both $\hat P^{n}\to \P^{B}_{X^{t,\chi,\nu}}$ and $\hat P^{n}\to  \P^{B}_{X^{t,\bar \chi,\bar \nu}}$ as $n\to \infty$.  \ep\\

\section{The stochastic target problem: alternative formulations and geometric dynamic programming principle}
\label{sectarget-DPP}

 Our aim is to provide a characterization of the set of initial measures  for the conditional law of the initial condition $\chi$ given $B$ such that the conditional law of $X^{t,\chi,\nu}_{T}$ given $B$ belongs to a fixed closed  subset $G$ of $\Pc_2$: 
 \beq\nonumber\label{defV_0}
\Vc(t)  & = &   \Big\{ \mu\in\Pc_2:~\exists (\chi,\nu)\in   \Xb^2_{t}\x \Uc~\mbox{s.t.}~\P^B_{\chi}=\mu~\mbox{ and }~\P^B_{X^{t,\chi,\nu}_T}\in  G  \Big\}.
\enq

In the above, and all over this paper, identities involving random variables must be taken in the a.s.~sense. In particular,  $\P^B_{X^{t,\chi,\nu}_T}\in  G$ means $\P^B_{X^{t,\chi,\nu}_T}\in  G ~~\P-{\rm a.s.}$

Before we go on, let us first give an example of application inspired from {agricultural crop management.}

\begin{Example}\label{example}  Consider the problem of a farmer that controls his production of wheat by spreading nitrogen fertilizer  or water on his field. The field is viewed as a collection of particles to which the farmer will bring additional fertilizer, water, etc. His aim is to maximize the dry mass level of the field, the quality of the wheat, etc., whose initial state can be viewed as a random variable $\chi$ (assigning $d$ characteristics of the production to each particle) over the two dimensional state space $\Omega^{\i}:=[0,1]^{2}$ modeling the field surface. The fertilizing effort is modeled by the control  $\nu$. Then, we let $X^{t,\chi,\nu}$ denote the current distribution of these characteristics. Its dynamics is of the form \eqref{eq: SDEMF}
in which the Brownian diffusion part is used to take into account several contingencies, e.g.~climatic ones. In particular, the dependency of the coefficients on $\P^{B}_{X^{t,\chi,\nu}}$ can model local interactions between particles (representing the points in the field), e.g.~related to the local water ressource, access to sun light, etc. The aim is to know what kind of initial state of the field allows to reach some given production level (in terms of volume, quality, etc.) at the end of the farming season. We shall come back to this example in Section \ref{sec: other formulation} below.
\end{Example}

\vspace{2mm}

We now show that $\chi$ in the definition of $\Vc(t)$  can be replaced by any random variable $\chi'\in \Xb^2_{t}$ such that  $\P^B_{\chi'}=\mu$. Apart from showing that only the distribution $\mu$ matters {(which is a desirable property if we think in terms of mass transportation)}, this will be of important use later on to provide a geometric dynamic programming principle for $\Vc$. 
 \begin{Proposition}\label{prop: equivWS} A measure $\mu\in\Pc_2$ belongs to $\Vc(t)$ if and only if for all $\chi\in \Xb^2_{t}$  such that $\P^B_{\chi}=\mu$ there exists $\nu\in \Uc$ for which  $\P^B_{X^{t,\chi,\nu}_T}\in  G$.
\end{Proposition}

\proof  Let $\tilde \Vc(t)$ denote the collection of measures $\mu\in\Pc_2$ such that  for all $\chi\in \Xb^2_{t}$  satisfying $\P^B_{\chi}=\mu$ there exists $\nu\in \Uc$ for which  $\P^B_{X^{t,\chi,\nu}_T}\in  G$. Clearly, $\tilde \Vc(t)\subset \Vc(t)$. 
We now prove the reverse inclusion. 
Let $\mu \in \Vc(t)$ and consider $(\chi,\nu)\in   \Xb^2_{t}\x \Uc$ such that $\P^B_{\chi}=\mu$ and $\P^B_{X^{t,\chi,\nu}_T}\in  G$. We fix $\bar\chi\in \Xb^2_{t}$ such that $\P^B_{\bar \chi}=\mu$ and we construct $\bar \nu\in\Uc$ such that $(\bar \chi,\bar \nu,B)$ and $(\chi,\nu,B)$ have the same law. Since $\P^B_{\chi}$ is deterministic, one can find a Borel map $\xr$ such that $\chi=\xr(\xi)$ $a.e.$

We first argue as in \cite[Proof of Proposition 3.1]{pham2018zero} and note that we can suppose $ \xr:~[0,1]^d\rightarrow\R^d$ to be surjective.  Indeed, if this is not the case, it is enough to modify $\xr$ on the set $\mathcal K\x \R^{d-1}$, where  $\mathcal K$ stands for the Cantor set, by the composition of  a surjective  map from $[0, 1]$ to $\R^d$  and $x\in \R^{d}\mapsto c(x^{1})$ where $c$ is  the Cantor function from $\Kc$ to $[0, 1]$. 
By \cite[Corollary 18.23]{AB06}, it follows that $ \xr$ admits an analytically measurable right-inverse, denoted by $\zeta : \R^d \rightarrow [0, 1]^d$, which satisfies

\begin{enumerate}[(i)]
\item $\xr(\zeta(x))=x$ for all $x\in\R^d$; 

\item $ \xr ^{-1}(\zeta^{-1}({A})) = {A}$, for any subset ${A}$ of $[0, 1]^d$;

\item $\zeta^{-1}({A})$ is analytically measurable in $\R^d$ for each Borel subset ${A}$ of $[0,1]^d$. 
\end{enumerate}
Recalling that every analytic subset of $\R^d$ is universally measurable (see e.g. Theorem 12.41 in \cite{AB06}), it follows that {one can find a Borel measurable map $\tilde \zeta$ such that $\zeta=\tilde \zeta$ Lebesgue almost-everywhere.}

We now define $\bar\xi$ by $\bar \xi=\tilde \zeta(\bar \chi)$, {so that $\bar \xi=\zeta(\bar \chi)$ a.e.} Since $\Fc_0$ is $\P$-complete, $\bar\xi$ is $\Fc_0$-measurable. Then using (ii) and since $\chi$ and $\bar \chi$ have the same law, we obtain
\beqs
\P(\bar \xi\in {A}) & = & \P(\bar \chi \in \zeta^{-1}({A}))~~=~~ \P(\chi \in \zeta^{-1}({A}))~~=~~\P^{\i}({A})\;,  \enqs
{for all Borel set $A$.} 
This proves that $\bar \xi$ has the same law as $\xi$.  Moreover, we have from (i)
\beqs
\xr(\bar \xi) & = & \bar \chi~{\P-{\rm a.s.}}
\enqs
which shows that $(\xi,\chi, B)$ and $(\bar \xi,\bar \chi, B)$ have the same law:
\begin{align*}
\P[\xi \in A_{1},\chi \in A_{2},B \in A_{3}]
&=\P[\xi \in A_{1},\xr(\xi) \in A_{2}]\P[B \in A_{3}]
\\
&=\P[\bar \xi \in A_{1},\xr(\bar \xi) \in A_{2}]\P[B \in A_{3}]
\\
&=\P[\bar \xi \in A_{1},\bar \chi \in A_{2},B \in A_{3}]
\end{align*}
for all Borel sets $A_{1},A_{2},A_{3}$. 
%

Since $\nu$ is $\F$-progressively measurable, it is, up to modification, of the form 
\beqs
\nu_s(\omega^{\circ},\omega^{\i}) & = & {\rm u}{(s,{}^s\!}B(\omega^{\circ}),\xi(\omega^{\i}))\;,\quad {s\in[t,T]}\;,~
\enqs
with ${\rm u}$ a Borel map.   
  Set now $\bar \nu:=u\1_{[0,t)}+\1_{[t,T]}{\rm u}(\cdot,{}^{\cdot}B,\bar \xi)\in \Uc$, for some $u\in \Ur$.

  Then, $(\bar \chi,\bar \nu_{t\vee \cdot}, B)$ and $(  \chi,  \nu_{t\vee \cdot}, B)$ have the same law, and Proposition \ref{prop: equiv loi jointe} implies  that $\P^B_{X^{t,\chi,\nu}_T}=\P^B_{X^{t,\bar \chi,\bar \nu}_T}$ so that the latter belongs to $G$, thus proving that $\Vc(t)\subset \tilde \Vc(t)$, by arbitrariness of $\bar \chi$.  \ep \\

 Before stating the dynamic programming principle, let us provide the following measurable selection lemma. We define the subset $\Gc$ of $[0,T]\times\L_2(\Omega^{\i},\Fc^{\i},\P^{\i};\R^d)$ by
\beqs
\Gc & := & \big\{(t,\chi)\in[0,T]\times\L_2(\Omega^{\i},\Fc^{\i},\P^{\i};\R^d):~\exists \nu \in\Uc\mbox{ s.t. }~\P^B_{X^{t,\chi,\nu}_{T}}\in G\big\}\;.
\enqs
From now on, we consider $\Uc$ as a subset of $\L_{2}([0,T]\x {\Omega}, dt\times d\P;\Ur)$ endowed with its strong topology.  We also introduce the subset $\Uc_t$ of $\Uc$ defined by
\beqs
\Uc_t & = & \big\{ \nu \in \Uc~:~\nu \mbox{ is  progressively measurable w.r.t }{\F}_{[t,T]}\big\}
\enqs
where ${\F}_{[t,T]}$ is the completion of $(\sigma((B_{r\vee t}-B_{t})_{0\le r\le s }, \xi))_{s\in [0,T]}$. We first rewrite the set $\Gc$ as follows.
\begin{Lemma}\label{lem : dans Ut} We have the following identification 
\beqs
\Gc & := & \big\{(t,\chi)\in[0,T]\times\L_2(\Omega^{\i},\Fc^{\i},\P^{\i};\R^d):~\exists \nu \in\Uc_t\mbox{ s.t. }~\P^B_{X^{t,\chi,\nu}_{T}}\in G\big\}\;.
\enqs
\end{Lemma}
\proof Let $\nu\in \Uc$ be such that $\P^B_{X^{t,\chi,\nu}_{T}}\in G$. Then, there exists  a progressively measurable map  $\ur$  such that $\nu_{s}(\omega)=\ur_{s}(\omega^{\circ},\omega^{\i})$ for $s\in [0,T]$. For $s\in [0,T]$, ${\rm w},{\rm w'}\in \Omega^{\circ}$, set ${\rm w}\oplus_{s}{\rm w'}:={\rm w}_{\cdot \wedge s} + ({\rm w'}_{\cdot \vee s}-{\rm w'}_{s})$. Define  $\nu_{s}^{ \omega^{\circ}}(\tilde \omega^{\circ},\omega^{\i}):=\ur_{s}(\omega^{\circ}\oplus_{t}\tilde \omega^{\circ},\omega^{\i})$.   Then,  one can find $\omega^{\circ}\in \Omega^{\circ}$ such that  $\P^B_{X^{t,\chi,\nu^{ \omega^{\circ}}}_{T}}(\tilde \omega^{\circ})\in G$ for $\P^{\circ}$-a.e.~$\tilde \omega^{{\circ}} \in \Omega^{\circ}$, see \cite[Theorem 5.4]{claisse2013note} and Proposition \ref{prop: eds def condi omega 1}. The control $\nu^{ \omega^{\circ}}$ is progressively measurable w.r.t. ${\F}_{[t,T]}$. 
\ep

\begin{Lemma}\label{Lem-meas-selec} For any probability measure $\mathfrak{P}$ on $[0,T]\times\L_2(\Omega^{\i},\Fc^{\i},\P^{\i};\R^d)$, there exists a measurable map $\vartheta:~\Gc\rightarrow \Uc$ such that
\beqs
\P^B_{X^{t,\chi,\vartheta(t,\chi)}_{T}} & \in & G 
\enqs
for $\mathfrak{P}$-a.e. $(t,\chi)\in \Gc$. Moreover, for each $(t,\chi)\in \Gc$, $\vartheta(t,\chi)$ can be chosen to be in $\Uc_t$. 
\end{Lemma}

\proof  It follows from \reff{eq: continuity loi par rapport condi initial} of Proposition \ref{prop: existence unicite eds} that the set 
$$
{\cal J}:=\{(t,\chi,\nu)\in   [0,T]\times\L_2(\Omega^{\i},\Fc^{\i},\P^{\i};\R^d)\x \Uc:~ \P^B_{X^{t,\chi,\nu}_{T}}\in G\mbox{ and }\nu\in \Uc_t\}
$$
is closed.  Moreover, the set $[0,T]\times\L_2(\Omega^{\i},\Fc^{\i},\P^{\i};\R^d)\x \Uc$ is a Polish space. Then,  the Jankov-von Neumann Theorem (see \cite[Proposition 7.49]{BertsekasShreve.78}),  ensures 
the existence of an analytically measurable function 
\begin{eqnarray*}
\tilde \vartheta:  [0,T]\times\L_2(\Omega^{\i},\Fc^{\i},\P^{\i};\R^d) & \longrightarrow & \mathcal{U}
\end{eqnarray*} 
such that
$$
(t,\chi, \tilde \vartheta(t,\chi))\in {\cal J} \; \mbox{ for   all }   (t,\chi) \in \Gc\;.
$$
Since any analytically measurable map is also universally measurable, the existence of $\vc$ follows from   \cite[Lemma 7.27]{BertsekasShreve.78}. We conclude by appealing to Lemma \ref{lem : dans Ut}.
\ep\\

We can now state the dynamic programming principle. In the following,   $\P^B_{X^{t,\chi,\nu}_{\theta}}\in  \Vc(\theta)$  means 
\beqs
\P^\circ\big(\big\{\omega^\circ\in \Omega^\circ~:~\P^B_{X^{t,\chi,\nu}_{\theta}}(\omega^\circ)\in  \Vc(\theta(\omega^\circ))   \big\}\big) & = & 1\;.
\enqs 

\begin{Theorem}\label{Th-DDP}
Fix $t\in[0,T]$ and $\theta\in \bar \Tc^{\circ}$ with values in $[t,T]$.  
Then,
\beqs
\Vc(t)  =  \Big\{ \mu\in\Pc_2:~\exists (\chi,\nu)\in   \Xb^2_{t}\x \Uc\mbox{ s.t. }~\P^B_{\chi}=\mu~\mbox{ and }~\P^B_{X^{t,\chi,\nu}_{\theta}}\in  \Vc(\theta) \Big\}.
\enqs 
\end{Theorem}
\begin{proof} Denote by $\hat \Vc(t)$ the right hand side of the equality in Theorem \ref{Th-DDP}. 

\noindent{1.} We first prove the inclusion $\Vc(t)\subset\hat \Vc(t)$. Fix $\mu\in \Vc(t)$. Then, there exists  $(\chi,\nu)\in   \Xb^2_{t}\x \Uc$  and $\tilde \Omega^{\circ}\in \Fc^{\circ}$ such that  $\P^{\circ}(\tilde\Omega^{\circ})=1$, $\P^B_{\chi}=\mu$ and $\P^B_{X^{t,\chi,\nu}_{T}}\in  G$ on $\tilde \Omega^{\circ}$.   
For $\tilde \omega^{\circ}\in \tilde \Omega^{\circ}$, we   define $(\chi^{\tilde \omega^{\circ}},\nu^{\tilde \omega^{\circ}})$ by 
\beqs
\chi^{\tilde \omega^{\circ}}(\omega) = X^{t,\chi,\nu}_{\theta(\tilde\omega^{\circ})}(\tilde\omega^{\circ},\omega^{\i}) &,&
\nu_s^{\tilde \omega^{\circ}}(\omega) = \nu_s(\tilde\omega^{\circ}\oplus_{\theta(\tilde\omega^{\circ})}\omega^{\circ},\omega^{\i})\;,\quad s\in[0,T]
\enqs
for all $\omega=(\omega^{\circ},\omega^{\i})\in\Omega$. Note  that $\chi^{\tilde \omega^{\circ}}\in\Xb^2_{\theta(\tilde\omega^{\circ})}$, $\P^B_{\chi^{\tilde \omega^{\circ}}}=\P^B_{X_{\theta}^{t,\chi,\nu}}(\tilde \omega^{\circ})$ and $\nu^{\tilde \omega^{\circ}}\in\Uc$ for all $\tilde\omega^{\circ}\in\tilde \Omega^{\circ}$.
Moreover, it follows from  \cite[Theorem 5.4]{claisse2013note} and Proposition \ref{prop: eds def condi omega 1} that 
$
X_T^{\theta(\tilde \omega^{\circ}), \chi^{\tilde \omega^{\circ}},\nu^{\tilde \omega^{\circ}}}$ has the same law as   
$X_T^{t,\chi,\nu}$ given $B_{\cdot \wedge \theta}=\tilde \omega^{\circ}_{\cdot \wedge \theta(\tilde\omega^{\circ})}$, for $\P^{\circ}$-a.e.~$\tilde \omega^{\circ}\in \Omega^{\circ}$. Since $\P^{B}_{X_T^{t,\chi,\nu}}(\omega^{\circ})\in G$ for $\omega^{\circ}\in \tilde \Omega^{\circ}$, it follows that 
$\P^B_{X_{\theta}^{t,\chi,\nu}}(\tilde \omega^{\circ})$ $=$ $\P^B_{\chi^{\tilde \omega^{\circ}}}\in \Vc(\theta(\tilde \omega^{\circ}))$  
 for all $\tilde \omega^{\circ}\in \tilde\Omega^{\circ}$. Therefore $\mu\in \hat \Vc(t)$.
 
\noindent{2.} We now prove the inclusion $\hat\Vc(t)\subset \Vc(t)$.
Fix   $\mu\in \hat \Vc(t)$ and  $(\chi,\nu)\in   \Xb^2_{t}\x \Uc$ such that  $\P^B_{\chi}=\mu$ and  $\P^B_{X^{t,\chi,\nu}_{\theta}}\in \Vc(\theta)$. {It follows from  Proposition \ref{prop: equivWS}  that $\big(\theta(\omega^{\circ}), X^{t,\chi,\nu}_{\theta(\omega^{\circ})}(\omega^{\circ},.)\big)\in \Gc$, for $\P^{\circ}$-a.e.~$\omega^{\circ}\in \Omega^{\circ}$.} Let $\mathfrak{P}$ be the probability measure   induced by $\omega^{\circ}\mapsto \big(\theta(\omega^{\circ}), X^{t,\chi,\nu}_{\theta(\omega^{\circ})}(\omega^{\circ},.)\big)$ on $[0,T]\times\L_2(\Omega^{\i},\Fc^1,\P^{\i};\R^d)$.
{By   
Lemma \ref{Lem-meas-selec}}, there exists a measurable map $\vartheta$ such that
$ \P^B_{X^{t',\chi',{\vartheta(t',\chi')}}_{T}} \in G $ $\P^{\circ}$-a.s.~for $\mathfrak{P}$-a.e.~$(t',\chi')\in \Gc$. Since  $\vc(t',\chi')$ can be chosen in the filtration 
$\F_{[t',T]}$ to which $^{t'}\!B$ is independent, $\P^B_{X^{t',\chi',{\vartheta(t',\chi')}}_{T}}$  is measurable with respect to  $\sigma(B_{\cdot\vee t'}-B_{t'})$.  
Hence, there exist null sets $N$ and $\tilde N$ such that 
$$
\P^B_{X^{\alpha(\omega^{\circ},\cdot) }_{T}}(\tilde \omega^{\circ}) \in G~\;\; \mbox{ for $\omega^{\circ} \notin N$ and $\tilde \omega^{\circ}\notin \tilde N$,}
$$
where 
$$
\alpha(\omega^{\circ},\cdot):= (\theta(\omega^{\circ}),X^{t,\chi,\nu}_{\theta}(\omega^{\circ},\cdot),\vartheta(\theta(\omega^{\circ}),X^{t,\chi,\nu}_{\theta}(\omega^{\circ},\cdot)).
$$
It remains to define the process $\bar \nu\in \Uc$ by
\beq\label{eq: concat control}
\bar\nu(\omega) & = & \nu(\omega)\mathds{1}_{[0,\theta(\omega^{\circ}))}+\vartheta(\theta(\omega^{\circ}),X^{t,\chi,\nu}_{\theta}(\omega^{\circ},\cdot))(\omega) \mathds{1}_{[\theta(\omega^{\circ}),T]}\;, 
\enq
and observe that  $X^{{\alpha}}_{T}=X^{{t,\chi,\bar \nu}}_{T}$, to conclude that   $\mu\in \Vc(t)$.
%
\end{proof}

\section{The dynamic programming partial differential equation}
\label{PDEdyn}
Let $\vr:~[0,T]\times\Pc_2\rightarrow \R$ be the indicator function of the complement of the reachability set $\Vc$:
\beq\label{defv}
\vr(t,\mu) & = & 1-\mathds{1}_{\Vc(t)}(\mu)\;,\quad (t,\mu)\in[0,T]\times\Pc_2.
\enq
The aim of this section is to provide a characterization of   $\vr$ as a (discontinuous) viscosity solution of a fully non-linear second order parabolic partial differential equation, in the spirit of \cite{soner2002dynamic}. Given Theorem \ref{Th-DDP}, this follows from combining the technologies developped in   \cite{carda12,CCD15}  and  \cite{soner2002dynamic}. We refer to Section \ref{sec: other formulation} for the specific case where the reachability set is an half-space in one direction.  

\subsection{Derivatives on the space of probability measures and It\^{o}'s lemma}\label{sec: Ito}

We first recall here   the notion of derivative with respect to a probability measure that has been introduced by Lions, see the lecture notes \cite{carda12}, and further developed in \cite{CCD15}, to our context.

We let $\tilde \Omega^{\i}$ be a polish space, $\tilde \Fc^{\i}$ its Borel $\sigma$-algebra and $\tilde \P^{\i}$ an atomless probability measure on $(\tilde \Omega^{\i},\tilde \Fc^{\i})$. We recall that we have $\Pc_2=\{{\tilde \P^{1}_{Y}}:=\tilde \P^{\i}\circ Y^{-1}~:~Y\in \L_{2}( \tilde\Omega^{\i}, \tilde\Fc^{\i},\tilde \P^{\i};\R^d)\}$. 

For a function $w:~\Pc_2\rightarrow\R$, we define its lifting as the function $W$ from $\L_{2}( \tilde\Omega^{\i}, \tilde\Fc^{\i}, \tilde\P^{\i};\R^{d})$ to $\R$ such that
\beqs
W(X) & = & w(\tilde\P^{\i}_X)\;,\quad \mbox{ for all }X\in \L_{2}(\tilde \Omega^{\i}, \tilde\Fc^{\i}, \tilde\P^{\i};\R^d)\,.
\enqs
 
We then say that $w$ is Fr\'echet differentiable (resp.~$\Cc^{1}$) on $\Pc_2$ if its lift $W$ is (resp.~continuously) Fr\'echet differentiable  on $\L_{2}(\tilde \Omega^{\i}, \tilde\Fc^{\i},\tilde \P^{\i};\R^{d})$.
If it exists, the Fr\'echet derivative $D W(X)$ of $W$ at   $X\in \L_{2}( \tilde\Omega^{\i}, \tilde\Fc^{\i}, \tilde\P^{\i};\R^{d})$ can be identified by Riez Theorem to an element of  $\L_{2}( \tilde\Omega^{\i}, \tilde\Fc^{\i}, \tilde\P^{\i};\R^{d})$   and admits a representation of the form 
\beq\label{eq: DW}
D W(X) & = & \partial_\mu w (\tilde\P^{\i}_X)(X)
\enq
for some  measurable map $\partial_\mu w ({\tilde \P^{\i}_X}):~\R^d\rightarrow\R^d$, that we call the derivative of $w$ at ${\tilde \P^{\i}_X}$ and we have $\partial_\mu w(\mu)\in \L_{2}(\R^d,\Bc(\R^d),\mu;\R^d)$ for  $\mu\in\Pc_2$.  In the case where $x\in \R^{d}\mapsto \partial_\mu w (\mu)(x)$ is differentiable at $x$, given $\mu \in \Pc_{2}$,  we denote by $\partial_{x}\partial_\mu w (\mu)(x)$ the corresponding gradient. 
 
Following \cite[Section 3.1]{CCD15}, we say that $w$  is fully $\Cc^2$ if it is $\Cc^1$ on $\Pc_2$ and
\begin{itemize}
\item the map $(\mu,x) \mapsto \partial_\mu w (\mu)(x)$ is continuous at any $(\mu,x)\in \Pc_2\times\R^d$,

\item for any $\mu\in \Pc_2$, the map $x \mapsto \partial_\mu w (\mu)(x)$ is continuously differentiable and the map $(\mu,x) \mapsto \partial_x\partial_\mu w (\mu)(x)$ is continuous at any $(\mu,x)\in \Pc_2\times\R^d$,  
\item for any $x\in\R^d$, the map $\mu\mapsto\partial_\mu w(\mu)(x)$ is differentiable in the lifted sense and its derivative,    regarded as the map $(\mu,x,x')\mapsto \partial^2_\mu w(\mu)(x,x')$, is continuous at any $(\mu,x,x')\in \Pc_2\times\R^d\times\R^d$.
\end{itemize}

From now on, we define $\Cc^{1,2}([0,T]\times \Pc_2)$
as the set of continuous functions $w:~[0,T]\times \Pc_2\rightarrow\R$ such that $w(t,\cdot)$ is fully $\Cc^{2}$ for all $t\in[0,T]$,  $\partial_t w$ exists and is continuous on $[0,T]\times \Pc_2$, $\partial_\mu w$, $\partial_x\partial_\mu w$ and $\partial^2_\mu w$ are continuous respectively on $[0,T]\times\Pc_2\times\R^d$, $[0,T]\times\Pc_2\times\R^d$ and $[0,T]\times\Pc_2\times\R^d\times\R^d$. We also define  $\Cc^{1,2}_{b}([0,T]\times \Pc_2)$ as the set of functions $w\in \Cc^{1,2}([0,T]\times\Pc_2)$ such that 
\beq\label{eq: cond inetgra derivee w time}
\sup_{t\in[0,T],\;\mu\in\Kc} \Big\{\big|\partial_t w(t,\mu)\big|+\int_{\R^d}\hspace{-1mm}\big|\partial_\mu w(t,\mu)(x)\big|^{2}d \mu(x) \quad\quad& & \nonumber \\
+\int_{\R^d}\big|\partial_x\partial_\mu w(t,\mu)\big|^2{d\mu(x)}\quad\quad\quad & & \nonumber\\
 + \int_{\R^d\times\R^d}|\partial^2_\mu w(t,\mu)(x,x')\big|^2d(\mu\otimes\mu)(x,x') \Big\} &  < &  \infty
\enq
for any compact subset $\Kc$ of $\Pc_2$.

\vspace{2mm}

We {are} now {in position to} derive a chain rule for the flow of conditional marginal laws of the controlled process. To this end, we introduce  the probability space  $(\tilde \Omega,\tilde \Fc,\tilde \P)$ defined by
\beq\label{def Omega tilde}
\tilde \Omega=\Omega^{\circ}\times\tilde \Omega^{\i}\;,\quad \tilde \Fc=\Fc^{\circ}\otimes\tilde\Fc^{\i} & \mbox { and } & \tilde \P=\P^{\circ}\otimes\tilde \P^{\i}. 
\enq
 As for the space $(\Omega,\Fc,\P)$, we denote by $\tilde\E_{B}$ the regular conditional expectation given $B$ on $(\tilde \Omega,\tilde \Fc,\tilde \P)$.

 \begin{Proposition}\label{prop: chain-rule}  Let $w\in \Cc^{1,2}_{b}([0,T]\times \Pc_2)$. Given $(t,\chi,\nu)\in [0,T]\x \Xb_{t}\x \Uc$, set $X=X^{t,\chi,\nu}$, $a=a(X_{},\P^{B}_{X_{}},\nu_{})$ and $b=b(X_{},\P^{B}_{X_{}},\nu_{})$. Then, 
\begin{align}
w(s,\P_{X_{s}}^{B})  &=  w(t,\P_{\chi}^{B})\nonumber\\
&+ \int_{t}^{s} \E_{B}\left[\partial_{t}w (r,\P_{X_{r}}^{B})+\partial_{\mu}w (r,\P_{X_{r}}^{B})(X_{r}) b_{r}\right] dr\nonumber\\
&  +\frac1{2} \int_{t}^{s} \E_{B}\left[\Tr\left(\partial_{x}\partial_{\mu}w (r,\P_{X_{r}}^{B})(X_{r}) a_{r} a^{\top}_{r}\right)\right] dr \nonumber\\
& +\frac{1}{2}  \int_{t}^{s} \E_{B}\left[\tilde\E_{B}\left[\Tr\left( \partial^{2}_{\mu}w (r,\P_{X_{r}}^B)(X_{r},\tilde X_{r}) a^{}_{r}\tilde a_{r}^{\top}\right)\right]\right] dr\nonumber\\
& + \int_{t}^{s} \E_{B}\left[\partial_{\mu}w (r,\P_{X_{r}}^{B})(X_{r}) a_{r}(X_{r},\P^{B}_{X_{r}},\nu_{r}) )\right] dB_{r} 
\nonumber
\end{align}
for all $s\in[t,T]$, where\footnote{{This means that $(\tilde X,\tilde a)(\omega^{\circ},\cdot)$, defined on $\tilde \Omega^{\i}$, has the same law as  $(X,a)(\omega^{\circ},\cdot)$, defined on $\Omega^{\i}$, for a.e.~$\omega^{\circ}\in \Omega^{\circ}$.}} $(\tilde X,\tilde a)$ is a copy of $(X,a)$ on $(\tilde \Omega,\tilde \Fc,\tilde \P)$.  
\end{Proposition}

\proof The proof follows from similar arguments as in \cite{CCD15} and we only mention the main ideas. 

We first define on $\tilde \Omega^{\i}$ a sequence of $i.i.d.$ random variables $(\xi^{\ell})_{\ell\ge 2}$ following the uniform law on $[0,1]^d$ (such a sequence exists since $\tilde \Omega^{\i}$ is polish and $\tilde \P^{\i}$ is atomless). 
 We then extend $B$, $\xi$ and $\xi^\ell$, $\ell\geq 2$ to $(\hat \Omega =\Omega^\circ\times \Omega^{\i}\times \tilde \Omega^{\i},\hat \Fc=\Fc^\circ\otimes\Fc^{\i}\otimes\tilde \Fc^{\i},\hat\P=\P^\circ\otimes\P^{\i}\otimes\tilde \P^{\i})$ in a canonical way by setting 
\beqs
\xi^{1}(\hat\omega)=\xi(\hat \omega)=\omega^{\i}\,,\;\xi^\ell(\hat\omega)  =  \xi^\ell(\tilde \omega^{\i})\;\mbox{ and } \;
B(\hat\omega)  & = & \omega^{\circ},
\enqs
for all $\hat\omega=(\omega^\circ,\omega^{\i},\tilde \omega^{\i})$. Note that $(\xi^{\ell})_{\ell \geq 1}$ is then an $i.i.d.$ sequence,  independent of $B$.

Since $\chi \in \Xb_{t}$ and $\nu\in \Uc$, we can find Borel maps $\xr$ and $\ur$ such that  $\chi=\xr(B, \xi^1)$ $\P$-a.s. and $\nu=\ur(\cdot, ^{\cdot}B, \xi^1)$, up to modification. 
We then set  $(\chi^{\ell},\nu^{\ell}):=(\xr(\xi^\ell), \ur(\cdot, ^{\cdot}B, \xi^{\ell}))$, for $\ell\ge 1$, 
and  define $X^{\ell}$ as the solution on $[t,T]$ of 
\begin{align*}
X^{\ell} =& \chi^{\ell} +\int_{  t}^{  \cdot} b^\ell_{s} ds
 + \int_{t}^{  \cdot} a^\ell_{s} dB_{s} ,
\end{align*}
in which  $(b^{\ell},a^{\ell})$ $=$ $(b,a)(X^{\ell}, \P_{X^{1}}^B, \nu^{\ell})$. It follows from Proposition \ref{prop: eds def condi omega 1} that   $(X^{\ell}_{r})_{\ell \ge 1}$ is a sequence of i.i.d.~random variables given $(B_{r'})_{r'\leq T}$, for each $ r \in[t, s]$. 
Set $\bar \mu^{N}_{r}:=\frac1N \sum_{\ell=1}^{N} \delta_{X^{\ell}_{r}}$ for $t\le r\le s$.

\noindent {1.} We first assume that $w\in\Cc_b^{1,2}([0,T]\times\Pc_2)$ is such that 
\beqs
(\mu,x,x')\mapsto(\partial_\mu w(\mu)(x),\partial_x\partial_\mu w(\mu)(x),\partial^{2}_\mu w(\mu)(x,x'))
\enqs
 is continuous,  and that
 $w$, $\partial_{\mu}w$, $\partial_{x}\partial_{\mu}w$ and $\partial^{2}_{\mu}w$ are bounded and uniformly continuous.
Then, it follows from \cite[Proposition 3.1]{CCD15} combined with It\^{o}'s Lemma that  
\beqs
w(s,\bar \mu^{N}_{s}) & = & w(t,\bar \mu^{N}_{t}) + \int_{t}^{s} \partial_{t}w (r,\bar \mu^{N}_{r}) dr
 +\frac1N \sum_{\ell=1}^{N} \int_{t}^{s} \partial_{\mu}w (r,\bar \mu^{N}_{r})(X^{\ell}_{r}) b^{\ell}_{r} dr\\
& & +\frac1N \sum_{\ell=1}^{N} \int_{t}^{s} \partial_{\mu}w (r,\bar \mu^{N}_{r})(X^{\ell}_{r}) a^{\ell}_{r}  dB_{r}
\\
& &
 +\frac1{2N} \sum_{\ell=1}^{N} \int_{t}^{s} \Tr\left[\partial_{x}\partial_{\mu}w (r,\bar \mu^{N}_{r})(X^{\ell}_{r}) a^{\ell}_{r}(a^{\ell }_{r})^{\top}\right] dr\\
& &+\frac1{2N^{2}} \sum_{\ell,n=1}^{N} \int_{t}^{s} \Tr\left[ \partial^{2}_{\mu}w (r,\bar \mu^{N}_{r})(X^{\ell}_{r},X^{n}_{r}) a^{\ell}_{r}(a^{n }_{r})^{\top}\right] dr.
\enqs
We now take the expectation given $(B_{r'})_{r'\leq T}$ on both sides and use \cite[Corollaries 2 and 3 of Theorem 5.13]{lip01} and \cite[Lemma 14.2]{kur} together with the fact that the quadruplets $(\bar \mu^{N}_{r},X^{\ell}_{r},X^{n}_{r},b^{\ell}_{r},b^{n}_{r},a^{\ell}_{r},a^{n}_{r})$, ${\ell,n \le N}$, have all the same law given $(B_{r'})_{r'\leq T}$, for $t\le r\le s$,  to obtain 
\beqs
\hat\E_{B}[w(s,\bar \mu^{N}_{s})] & = & \hat\E_{B}[w(t,\bar \mu^{N}_{t})]+  \int_{t}^{s} \hat\E_{B}\left[\partial_{t}w (r,\bar \mu^{N}_{r}) +\partial_{\mu}w (r,\bar \mu^{N}_{r})(X^{1}_{r}) b^{1}_{r}\right] dr\\
 & &+  \int_{t}^{s} \hat\E_{B}\left[\partial_{\mu}w (r,\bar \mu^{N}_{r})(X^{1}_{r}) a^{1}_{r} )\right] dB_{r}\\
& &+\frac1{2} \int_{t}^{s} \hat\E_{B}\left[\Tr\left(\partial_{x}\partial_{\mu}w (r,\bar \mu^{N}_{r})(X^{1}_{r}) a^{1}_{r}(a^{1}_{r})^{\top}\right)\right] dr\\
& &+\frac1{2N }  \int_{t}^{s} \hat\E_{B}\left[\Tr\left( \partial^{2}_{\mu}w (r,\bar \mu^{N}_{r})(X^{1}_{r},X^{1}_{r}) a^{1}_{r}(a^{1}_{r})^{\top}\right)\right] dr
\\
& & +\frac{N-1}{2N }  \int_{t}^{s} \hat\E_{B}\left[\Tr\left( \partial^{2}_{\mu}w (r,\bar \mu^{N}_{r})(X^{1}_{r},X^{2}_{r}) a^{1}_{r}(a^{2}_{r})^{\top}\right)\right] dr,
\enqs
where $\hat\E_{B}$ stands for the condition expectation given  $(B_{r'})_{r'\leq T}$ on $\hat \Omega$.
We then use the fact that $\Wc_2(\bar \mu^{N}_{r},\P^B_{X^{1}_{r}})\to0 $ a.s. as $N\to \infty$ for all $ r\in[t, s]$. This is a consequence of \cite[Lemma 4]{JMW10} and the fact that $(X^{\ell}_{r})_{\ell \ge 1}$ is a sequence of i.i.d.~random variables given $(B_{r'})_{r'\leq T}$. 
Since all the involved maps are assumed to be bounded and continuous, one can take the limit as $N\to \infty$ in the above  to obtain
\beq
w(s,\P_{X^{1}_{s}}^B) & = & w(t,\P_{\chi^{1}}^B)+ \int_{t}^{s} \E_{B}\left[\partial_{t}w (r,\P_{X^{1}_{r}}^B)+\partial_{\mu}w (r,\P_{X^{1}_{r}}^B)(X^{1}_{r}) b^{1}_{r}\right] dr\nonumber\\
 & &+ \int_{t}^{s} \E_{B}\left[\partial_{\mu}w (r,\P_{X^{1}_{r}}^B)(X^{1}_{r}) a^{1}_{r} )\right] dB_{r} \label{eq: Ito proof}\\
& & +\frac1{2} \int_{t}^{s} \E_{B}\left[\Tr\left(\partial_{x}\partial_{\mu}w (r,\P_{X^{1}_{r}}^B)(X^{1}_{r}) a^{1}_{r}(a^{1}_{r})^{\top}\right)\right] dr\nonumber
\\
& & +\frac{1}{2}  \int_{t}^{s} \E_{B}\left[\tilde \E_{B}\left[\Tr\left( \partial^{2}_{\mu}w (r,\P_{X^{1}_{r}}^B)(X^{1}_{r},X^{2}_{r}) a^{1}_{r}(a^{2}_{r})^{\top}\right)\right]\right] dr.\nonumber
\enq
\noindent{2.} The validity of \reff{eq: Ito proof} can be extended to the case where $w$ is just in $\Cc^{1,2}_{b}([0,T]\times\Pc_2)$ by following the molifying argument of \cite[Proposition 3.4]{CCD15} whenever the condition  \reff{eq: cond inetgra derivee w time}  holds, recall that $(b,a)$ is bounded.
\ep

\vspace{2mm}

Later on, we shall need to use this It\^{o}'s formula at the level {of a map $W$  defined on $ \L_{2}(\tilde \Omega^{\i}, \tilde\Fc^{\i},\tilde\P^{\i};\R^{d})$}.  {When $W$ is the lift of a $\Cc^{1,2}_{b}$ function $w$, and under} the additional assumption that $W$ is twice continuously Fr\'echet differentiable\footnote{{Being  $\Cc^{1,2}_{b}$ for the function $w$ is not a sufficient condition for the lift $W$ to be twice Fr\'echet differentiable as shown  in \cite[Example 2.3]{BLPR17}. }}, $D^2W$ can be identified by Riez Theorem as a self-adjoint operator on $\L_{2}( \tilde \Omega^{\i},$ $ \tilde \Fc^{\i},\tilde \P^{\i};\R^d)$ and we have  
the following identification by \cite[Remark 6.4]{CD14} 
\beq\label{eq:id2ndder}
  \tilde \E^{\i}\left[D^2 W(X)(Y)Y^{\top}\right] & = &  \tilde \E^{\i}\left[ \Tr\left(\partial_x \partial_\mu w(\mu)(X)YY^\top\right)\right]\\
   & & + \tilde \E^{\i}\left[\tilde \E'^{\i}\left[\Tr\left( \partial^{2}_{\mu}w (\mu)(X,X') Y(Y')^\top\right)\right]\right] dr\nonumber
\enq
for any random variables $X\in \L_{2}(\tilde \Omega^{\i}, \tilde\Fc^{\i},\tilde\P^{\i};\R^{d})$ with $\tilde \P^{\i}_{X}=\mu$ and  $Y\in \L_{2}(\tilde \Omega^{\i},\tilde \Fc^{\i},\tilde\P^{\i};\R^{d})$, where $(X',Y')$ is a copy of $(X,Y)$ on another Polish atomless probability space $ (\tilde\Omega'^{\i},\tilde \Fc'^{\i},\tilde\P'^{\i})$, and $\tilde \E'^{\i}$ is the expectation operator under $\tilde\P'^{\i}$.

{Let us say} that $W:[0,T]\times \L_{2}( \tilde \Omega^{\i},\tilde  \Fc^{\i}, \tilde \P^{\i};\R^{d})\to \R$ is $\Cc^{1,2}_{b}$  if it is the lifting function of a map $w\in \Cc^{1,2}_{b}([0,T]\times \Pc_2)$. Given a random variable $X\in \Lb_{2}(\tilde\Omega,\tilde\Fc,\tilde\P;\R^{d})$ (recall that $(\tilde \Omega, \tilde \Fc, \tilde \P)$ is defined in \reff{def Omega tilde}), we define $W(t,X)$ as the random variable $\omega^{\circ}\in \Omega^{\circ}\mapsto W(t,X(\omega^{\circ},\cdot))$ where $X(\omega^{\circ},\cdot)$ is now a random variable on $\L^{2}(\tilde \Omega^{\i},\tilde \Fc^{\i},\tilde \P^{\i};\R^{d})$. We use the same convention for $DW(t,X(\omega^{\circ},\cdot))$ and $D^{2}W(t,X(\omega^{\circ},\cdot))$.  For $(t,\chi,\nu)\in [0,T]\times\Xb_t\times \Uc$, we introduce $\tilde \chi, \tilde \nu$ copies of $ \chi, \nu$ defined on $\tilde \Omega$ and we define the process $\tilde X$ on  $\tilde \Omega$ solution  to \reff{eq: SDEMF} with initial conditions $(t,\tilde \chi)$ and control $\tilde \nu$.
 As an immediate corollary of Proposition \ref{prop: chain-rule} and \reff{eq:id2ndder}, we then have the following: 
\begin{align}
W(s, {\tilde X_{s}} )  &=  W(t,{\tilde \chi})\nonumber\\
&+ \int_{t}^{s} \tilde \E_{B}\left[\partial_{t}W (r,{\tilde X_{r}})+DW (r,{\tilde X_{r}}) b_{r}(\tilde X_{r},\tilde \P^{B}_{X_{r}},\tilde \nu_{r})\right] dr\nonumber\\
&  +\frac1{2} \int_{t}^{s} \tilde \E_{B}\left[D^{2}W (r,\tilde X_{r})(X_{r}) a_{r}a_{r}^{\top}(\tilde X_{r},\tilde \P^{B}_{\tilde X_{r}},\tilde \nu_{r})\right] dr\nonumber\\
 & + \int_{t}^{s} \tilde \E_{B}\left[DW (r,\tilde X_{r}) a_{r}(\tilde X_{r},\tilde \P^{B}_{\tilde X_{r}},\tilde \nu_{r}) )\right] dB_{r}, \label{eq: Ito on W}
\end{align}
for all $s\in[0,T]$,  whenever  $W$ is in   
$\Cc^{1,2}_{b}\cap C^{1,2}([0,T]\x\L_{2}(\tilde  \Omega^{\i}, \tilde \Fc^{\i}, \tilde \P^{\i};\R^{d}))$.

This result is in fact true even when $W$ is not necessarily the lift of a law-invariant map, but simply $C^{1,2}([0,T]\x\L_{2}(\tilde  \Omega^{\i}, \tilde \Fc^{\i}, \tilde \P^{\i};\R^{d}))$.

\begin{Proposition}\label{RK-nonlif-CR} Fix $W\in  C^{1,2}([0,T]\x\L_{2}(\tilde  \Omega^{\i}, \tilde \Fc^{\i}, \tilde \P^{\i};\R^{d}))$, then \eqref{eq: Ito on W} holds.
\end{Proposition} 

\proof  This follows from the proof of  \cite[Proposition 6.3]{CD14} up slight  adaptations  similar too the ones made in the Proposition \ref{prop: chain-rule}.
\ep

\subsection{{Verification argument}}

We recall that aim at characterizing the function $\vr: (t,\mu)\in[0,T]\x\Pc_2\mapsto1-\mathds{1}_{\Vc(t)}(\mu)$.
Following \cite{BLPR17,soner2002dynamic}, one can expect it to solve, in a certain sense, the PDE
\begin{equation}\label{eq: pde vr}
-\partial_t w(t, \mu) + H\big(t, \mu, \partial_{\mu} w(t, \mu), \partial_{\mu}\partial_{x}w(t, \mu),\partial^2_{\mu}w(t, \mu)\big)   =   0 \;, 
\end{equation}
in which 
\begin{align*}
H\big(t, \mu, \partial_{\mu} w(t, \mu), \partial_{\mu}\partial_{x}w(t, \mu),\partial^2_{\mu}w(t, \mu)\big):=
\sup_{u\in N(t,\mu,\partial_{\mu}w(t, \mu))} \left(- L_{t}^{u}[w](\mu)\right)
\end{align*}
with 
\begin{align*}
N(t,\mu,\partial_{\mu}w(t, \mu)):=&\left\{u\in \L_{0}(\R^{d};\Ur):\int \partial_{\mu}w (t,\mu)(x) a_{t}(x,\mu,u(x) )\mu(dx)=0\right\}
\end{align*}
where $\L_{0}(\R^{d};\Ur)$ stands for the collection of $\Ur$-valued Borel maps on $\R^{d}$, and 
\begin{align*}
&L_{t}^{u}[w](\mu)\\
&:=\int \int \left\{ b_{t}(x,\mu,u(x))^{\top}\partial_{\mu}w (t,\mu)(x)   
+\frac1{2}\Tr\left[\partial_{x}\partial_{\mu}w (t,\mu)(x) (a_{t} a^{\top}_{t})(x,\mu,u(x))\right] \right. \\
&~~~~\left.+\frac{1}{2}  \Tr\left[ \partial^{2}_{\mu}w (t,\mu)(x,\tilde x) a_{t}(x,\mu,u(x)) a_{t}^{\top}(\tilde x,\mu, u(\tilde x))\right]\right\}\mu(dx)\mu(d\tilde x). 
\end{align*}

There is however little chance that the above equation admits a smooth solution, and, as usual, we shall appeal to the notion of viscosity solutions, see Section \ref{subsec: visco cara} below. Still, one can check whether a measure $\mu$ belongs to the set $\Vc(t)$ by using a verification argument.\footnote{{We leave the study of more precise examples to future research.}} 

 \begin{Proposition} Let $w\in  \Cc^{1,2}_{b}([0,T]\times \Pc_2)$ and $u$ be a $U$-valued Borel map on $[0,T]\x \Omega^{\circ}\x \R^{d}$   which is  $\F$-{progressive}$\otimes\Bc(\R^d)$-measurable. Fix $t\le T$ and $\mu\in \Pc_2$ and assume that existence holds for \eqref{quenched-sde} with $\nu:=u(\cdot,X^{t,\chi,\nu}_{\cdot})$, for some $\chi \in \mathbf{X}_{t}$ such that $\P^B_\chi=\mu$. Assume further that 
 \begin{align*}
&  -\partial_t w(\cdot,\P^{B}_{X^{t,\chi,\nu}_{\cdot}}(\omega^{\circ}))- L_{\cdot}^{u(\cdot,\omega^{\circ},\cdot)}[w](\P^{B}_{X^{t,\chi,\nu}_{\cdot}}(\omega^{\circ}))\ge 0~dt-{\rm a.e.}
\\
& u(\cdot,\omega^{\circ},\cdot)\in N(\cdot,\P^{B}_{X^{t,\chi,\nu}_{\cdot}}(\omega^{\circ}),\partial_{\mu}w(\cdot, \P^{B}_{X^{t,\chi,\nu}_{\cdot}})(\omega^{\circ}))~dt-{\rm a.e.}
\\
&w(T,\cdot)\ge 1-\1_{G} \;\mbox{ on } \Pc_{2},
 \end{align*}
 for $\P^\circ$-almost all $\omega^{\circ}\in \Omega^{\circ}$.  
Then,   $\mu \in \Vc(t)$ whenever $w(t,\mu)\le 0$.
\end{Proposition}

\proof Our conditions ensure that $\nu \in \Uc$. Moreover, the chain rule of Proposition \ref{prop: chain-rule} combined with the above imply that $w(T,\P^{B}_{X^{t,\chi,\nu}_{T}})\le 0$. Hence,  $1-\1_{G}(\P^{B}_{X^{t,\chi,\nu}_{T}})\le 0$ so that $\P^{B}_{X^{t,\chi,\nu}_{T}}\in G$.\ep
\subsection{Viscosity solution characterization}\label{subsec: visco cara}

As already mentioned, we shall in general rely on the notion of viscosity solutions. For this, we need to work at the level of the  lifting function $V:[0,T]\x  \L_{2}(\tilde \Omega^{\i},\tilde \Fc^{\i},\tilde \P^{\i};\R^d)\to \R$ of  $\vr$. In view of \eqref{eq: DW}-\eqref{eq:id2ndder}, one expects that it solves on $[0,T)\x \L_{2}(\tilde \Omega^{\i},\tilde \Fc^{\i},\tilde \P^{\i};\R^d)$  
\begin{equation}\label{eq: EDPV}
-\partial_t W + \Hc \big(\cdot, D W, D^2W\big)  =  0 \;.
\end{equation}
where $\Hc$ is defined as $\Hc_{0}$ with, for  $\eps\ge 0$, 
\begin{align*}
\Lc^{u}_{t}(\chi,P,Q) &:=\tilde \E_{B} \Big[b_t^{\top}(\chi,\P_\chi,u)P+{1\over 2}Q \big(a_t(\chi,\P_\chi,u)Z\big)a_t(\chi,\P_\chi,u)Z\Big]
\\
\Hc_{\eps}(t,\chi,P,Q) & :=  \sup_{u\in\Nc_{\eps}(t,\chi,P)}\hspace{-1mm} \Big\{-\Lc^{u}_{t}(\chi,P,Q) \Big\}
\\
\Nc_{\eps}(t,\chi,P) & :=  \Big\{ u \in \L_{0}(\tilde \Omega,\tilde \Fc,\tilde \P;\Ur)~:~|\tilde \E_{B}[a_t(\chi,\P_\chi,u)P]|\le \eps \Big\},
\end{align*}
for  $t\in[0,T]$, $u\in \L_{0}(\tilde \Omega,\tilde \Fc,\tilde \P;\Ur)$, $\chi,P\in \L_{2}(\tilde \Omega,\tilde \Fc, \P;\R^{d})$ and $Q\in  S(\L_{2}(\tilde  \Omega,$ $\tilde \Fc, \tilde \P;\R^{d}))$, the set of self-adjoint operators on $\L_{2}(\tilde \Omega, $ $\tilde \Fc, $ $\tilde \P;\R^d)$.
 
Let us recall that $W: [0,T]\x  \L_{2}(\tilde \Omega^{\i},\tilde \Fc^{\i},\tilde \P^{\i};\R^d)$ is extended to $[0,T]\x  \L_{2}(\tilde \Omega,\tilde \Fc,\tilde \P;\R^d)$ by defining $W(t,X)$ as the random variable $\omega^{\circ}\in \Omega^{\circ}\mapsto W(t,X(\omega^{\circ},\cdot))$.

Since neither $V$ nor $\Hc_{\cdot}$ are a-priori continuous, we  define $V_{*}$ and $V^{*}$ as the lower-semicontinous and upper-semicontinuous enveloppes of $V$, and let $\Hc^{*}$ and $\Hc_{*}$ be defined as  the relaxed upper- and lower-semilimits as $\eps\to 0$.  

 We say that  $V_{*}$ is a viscosity supersolution (resp.~$V^{*}$ is a subsolution) of \reff{eq: EDPV} if  for any $(t,\chi)\in [0,T]\times \L_{2}(\tilde  \Omega^{\i}, \tilde \Fc^{\i},\tilde  \P^{\i};\R^{d})$ and any function $\Phi\in C^{1,2}\big([0, T ]\times \L_{2}( \tilde \Omega^{\i},\tilde  \Fc^{\i},\tilde  \P^{\i};\R^{d})\big)$   such that 
\beqs
(V_*-\Phi) (t,\chi) & = & \min_{[0,T]\times \L_{2}(\tilde  \Omega^{\i}, \tilde \Fc^{\i}, \tilde \P^{\i};\R^{d})}(V_*-\Phi) \\
\text{ ( resp.  } (V^*-\Phi) (t,\chi) & = & \max_{[0,T]\times \L_{2}( \tilde \Omega^{\i},\tilde  \Fc^{\i},\tilde  \P^{\i};\R^{d})}(V^*-\Phi) \text{ )}
\enqs
we have 
\beqs
-\partial_t \Phi(t, \chi) + \Hc^* \big(t, \chi, D \Phi(t, \chi), D^2\Phi(t,  \chi)\big)  & \geq &  0 \\
\text{ (resp. }-\partial_t \Phi(t, \chi) + \Hc_* \big(t, \chi, D \Phi(t, \chi), D^2\Phi(t,  \chi)\big)  & \leq &  0 \text { )} \;.
\enqs
If $V_{*}$ is a supersolution and $V^{*}$ is a subsolution, we say that $V$ is a discontinuous solution.
\\
 
We are now ready to state the viscosity property of the function $V$. This requires the following continuity assumption on the set $\Nc$.

\vspace{2mm}
\noindent\textbf{(H2)}: Let $\Oc$ be an open subset of $[0,T]\times \L_{2}(\tilde \Omega,\tilde \Fc,\tilde \P;\R^d)\x \L_{2}(\tilde \Omega,\tilde \Fc,\tilde \P;\R^d)$ such that  $\Nc_{0}(t,\chi,P)\neq\emptyset$ for all  $(t,\chi,P)\in\rm \Oc$. Then, for every $\eps>0$, $(t_0,\chi_0,P_0)\in  \Oc$ and $u_0\in \Nc_{0} (t_0,\chi_0,P_0)$, there exists an open neighborhood $\Oc'$ of $(t_0,\chi_0,P_0)$ and a measurable map $\hat u:[0,T]\x \R^{d} \x \R^{d}\x \tilde \Omega^{\i}$ $\to $ $\Ur$  such that:
\\
\noindent{\rm (i)} $\tilde \E_{B}[|\hat u_{t_{0}}(\chi_0,P_0,\xi)-u_0|]\leq \eps$.\\
\noindent{\rm (ii)} There exists $C>0$ for which 
\begin{align*}
\tilde \E[|\hat u_{t}(\chi, P,\xi)-\hat u_{t}(\chi', P',\xi)|^{2}]\le C\tilde \E[|\chi-\chi'|^{2}+|  P-P' |^2] 
\end{align*}
for all $(t,\chi,P),(t,\chi',P')\in \Oc'$.\\
\noindent{\rm (iii)} $\hat u_{t}(\chi, P,\xi)\in \Nc_{0}(t,\chi,P)$ $\P^{\circ}-a.e.$, for all $(t,\chi,P)\in \Oc'$.
 
 \vspace{2mm}

\ni We also strengthen \textbf{(H1)} by the following additional condition.

\vspace{2mm}

\noindent\textbf{(H1')} There exist a constant $C$ and a function $m:~\R_+\rightarrow\R$ such that 
$
m(t) \to  0
$ as $t\to 0$ 
and 
\beqs
|b_t(x,\mu,u)-b_{t'}( x,\mu,u')|+|a_t(x,\mu,u)-a_{t'}(x,\mu,u')| & \leq & m(t-t')+C|u-u'|. 
\enqs
for all $t,t'\in[0,T]$, $x\in\R^d$, $\mu\in \Pc_2$ and $u,u'\in \Ur$.

\vspace{2mm}

\begin{Theorem}  \label{THMBT}
Under \textbf{(H1)} and \textbf{(H1')} the function $V_{*}$ is a viscosity supersolution of \reff{eq: EDPV}.
If in addition \textbf{(H2)} holds, then $V^{*}$ is a viscosity subsolution of \reff{eq: EDPV}.
\end{Theorem}

\proof {{\bf Part I.}~Supersolution property.}  Fix $(t_0,\chi_0)\in [0,T)\times \L_{2}(\tilde \Omega^{\i},\tilde \Fc^{\i},\tilde \P^{\i};\R^d)$ and a  test function  $\Phi\in C^{1,2}\big([0,T)\times \L_{2}(\tilde \Omega^{\i},\tilde \Fc^{\i},\tilde \P^{\i};\R^d)\big)$ such that 
\beqs
( V_*-\Phi) (t_0,\chi_0) & = & \min_{[0,T]\times \L_{2}(\tilde \Omega^{\i},\tilde \Fc^{\i},\tilde \P^{\i};\R^d)}(V_*-\Phi)~~=~~0\;.
\enqs
We prove that
\beq\label{ppte-sur-sol}
-\partial_t \Phi(t_0, \chi_0) + \Hc^* \big(t_0, \chi_0, D \Phi(t_0, \chi_0), D^2\Phi(t_0, \chi_0)\big)  & \geq &  0\;. 
\enq 
\textbf{1.} Suppose that the function $V$ is constant in a neighborhood of $(t_0,\chi_0)$. 
 Then $\Phi(t_0,\chi_0)$ is a local maximum of $\Phi$ and therefore
\beq\label{cond-V-loc-cst}
\partial_t \Phi(t_0,\chi_0)\leq0\;,\quad D \Phi(t_0,\chi_0)=0 & \mbox{ and } & D^2 \Phi(t_0,\chi_0){\le }0\;. 
\enq
Hence, $ \Nc_{0}(t_0,\chi_0,D \Phi(t_0, \chi_0))= \L_{0}(\tilde\Omega,\tilde \Fc,\tilde \P;\Ur)$ and 
\beqs
-\partial_t \Phi(t_0, \chi_0) + \Hc_{0} \big(t_0, \chi_0, D \Phi(t_0, \chi_0), D^2\Phi(t_0, \chi_0)\big)  & \geq &  0\;, 
\enqs
  so that \reff{ppte-sur-sol} is satisfied.

\noindent\textbf{2.} We now consider the complementary case: $V_*(t_0,\chi_0)=0$. Let $(t_n,\chi_n)_{n\geq 1}$ be a sequence of $[0,T)\times \L_{2}(\tilde \Omega^{\i},\tilde  \Fc^{\i},  \tilde \P^{\i};\R^{d})$  converging to $(t_0,\chi_0)$ and such that 
\beq\label{cond v app seq}
V(t_n,\chi_n) & = & 0\;,\quad \mbox{ for all }~n\geq 1.
\enq
 We argue by contradiction and suppose that 
\beqs
-\partial_t \Phi(t_0, \chi_0) + \Hc^* \big(t_0, \chi_0, D \Phi(t_0, \chi_0), D^2\Phi(t_0, \chi_0)\big)  & =: & -2\eta  
\enqs
for some $\eta>0$. Define 
\beqs
\tilde \Phi(t,\chi) & = & \Phi(t,\chi)-\varphi\big(|t-t_0|^2+\E\big[\big|\chi-\chi_0\big|^2\big]^{2}\big)\;
\enqs
for $(t,\chi)\in [0,T]\times\L_{2}(\tilde \Omega^{\i},\tilde \Fc^{\i},\tilde \P^{\i};\R^d)$, where $\varphi\in C^\infty(\R,\R)$ is such that $\varphi(x)= x$ for $x\in[0,1]$ and $\varphi(x)= 2$ for $x\geq2$.  
Then,  
\begin{align*}
(\partial_t\tilde \Phi,D\tilde \Phi,D^2\tilde \Phi)(t_0,\chi_0) =(  \Phi,D  \Phi,D^2  \Phi)(t_0,\chi_0) ,
\end{align*}
and   
we can find $\eps>0$ and an open ball $B_{\eps}(t_{0},\chi_{0})$ such that  
\begin{align} 
 -\eta\ge& -\partial_t\tilde \Phi(t,\chi)-\Lc^{u}_{t}(\chi,D\tilde \Phi(t,\chi),D^{2}\tilde \Phi(t,\chi))  \;\label{cond-local}
\end{align}
for any $(t,\chi)\in B_\eps(t_0,\chi_0)$ and any $u\in \Nc_\eps(t,\chi,D\Phi(t,\chi))$. Let $\partial_pB_\eps(t_0,\chi_0) := \{t_0 + \eps\} \times cl (B_\eps(\chi_0)) \cup [t_0, t_0 + \eps) \times \partial B_\eps(\chi_0)$ denote the parabolic boundary of $B_\eps(t_0, \chi_0)$ and observe that
\beq\label{min-strict}
\zeta & := & \inf_{\partial_pB_\eps(t_0,\chi_0)} (V_*-\tilde \Phi)~~>~~0  \;.
\enq
In view of \reff{cond v app seq}, we can find  a control $\nu^n\in \Uc$ such that
\beqs
\tilde \P^B_{X^n_t} & \in & G\;,
\enqs
where $X^n= X^{t_n,\chi_n,\nu^n}$. We then define the stopping times 
\beqs
\theta_n(\omega^{\circ}) & = & \inf\Big\{s\geq t_n~:~\big(s,X^n_{s}(\omega^{\circ},.)\big)\notin B_\eps(t_0,\chi_0)\Big\}\;,\quad \omega^{\circ}\in \Omega^{\circ}\;.
\enqs
{By  Theorem \ref{Th-DDP},
$
V(\cdot ,X^n_\cdot )  = 0$ on $[t_{n},T]$, so that   $-\tilde \Phi(\cdot ,X^n )\ge 0$ on $[t_{n},T]$ and $-\tilde \Phi(\theta_n ,X^n_{\theta_n } )\ge  \zeta $ by \reff{min-strict}}.
Let us set $\beta_n:=-{\tilde \Phi}(t_n,\chi_n)$ and define 
\begin{align*}
\alpha^{n}_{t}:=& {\tilde \E_{B}}[
\partial_t\tilde \Phi(t,{X^n_t})
+\Lc^{\nu^{n}_{t}}_{t}({X^n_t},D\tilde \Phi(t,{X^n_t}),D^2\tilde \Phi(t,{X^n_t}))],\\
\rho^n :=& -  \alpha^{n}  \mathds{1}_{A_n}\,,\,
\psi^n  :=  -\tilde \E_B\big[a(X^n,\tilde \P^B_{X^n},\nu^n)D\tilde \Phi(\cdot,X^n)\big]
\end{align*}
with
\begin{align*}
A_n := & \Big\{ t\in [t_n,{\theta_n}]\;:\; -\alpha^{n}_{t}    >  -\eta \Big\}.
\end{align*}
Applying Proposition \ref{RK-nonlif-CR} to $\tilde \Phi(.,X^n)$, we then get
that $   M^n_{{\theta_{n}}}\ge 0$ 
{where }
\begin{align}\label{def-Mn}
 {M^n:=\beta_n-\zeta+\int_{t_n}^{\cdot} \rho^n_{t}dt+\int_{t_n}^{\cdot}\psi^n_{t}dB_{t}\ge \beta_n-\zeta\ge -\frac12 \zeta,}
\end{align}
{for $n$ large.}
By \reff{cond-local}, 
\beqs
\big|\tilde \E_B\big[a_t(X^n_t,\P^B_{X^n_t},\nu^n_t)D\tilde \Phi(t,X^n_t)\big]\big| & > & \eps\;,\;\mbox{ for } t\in A_n,
\enqs
and we can   define   the positive $\bar \F^{\circ}$-local martingale $L^n$ by 
\beqs
L^n_{t} & = & 1-\int_{t_n}^{t} L^n_{s}\rho^n_{s}|\psi^n_{s}|^{-2}\psi^n_{s}dB_{s} \;,\quad t\geq t_{n}\;.
\enqs
{The coefficients $a$ and $b$ being bounded, $L^{n}$ is a true martingale.}
{In view of \reff{def-Mn}, $L^nM^n$ is a non-negative local martingale that is bounded from below by a martingale.} Therefore, it is a super-martingale and  
\beqs
0 \leq  \E[L^n_{\theta_{n}}M^n_{\theta_{n}}]\leq L^n_{t_{n}}M^n_{t_{n}}=M^n_{t_{n}}=\beta_n-\zeta\;.
\enqs
Sending $n$ to $\infty$, we get a contradiction since $\beta_{n}\to 0$.

\noindent{{\bf Part II.}~Subsolution property.}  Fix $(t_0,\chi_0)\in [0,T)\times \L_{2}(\tilde \Omega^{\i},\tilde \Fc^{\i},\tilde \P^{\i};\R^d)$ and $\Phi\in C^{1,2}\big( [0,T]\times \L_{2}(\tilde \Omega^{\i},\tilde \Fc^{\i},\tilde \P^{\i};\R^d)\big)$ such that 
\beq\label{test cond sub sol}
(V^*-\Phi) (t_0,\chi_0) & = & \max_{[0,T]\times \L_{2}(\tilde \Omega^{\i},\tilde \Fc^{\i},\tilde \P^{\i};\R^d)}(V^*-\Phi).
\enq
We have to prove that
\beqs
-\partial_t \Phi(t_0, \chi_0) + \Hc_* \big(t_0, \chi_0, D \Phi(t_0, \chi_0), D^2\Phi(t_0, \chi_0)\big)  & \leq &  0\;. 
\enqs 
We   distinguish two cases.

\vspace{2mm}

\ni \textbf{1.} Suppose that $V^*(t_0,\chi_0)=0$. Then, we deduce from \reff{test cond sub sol} that 
\beq\label{cond-1-sub-sol}
\partial_t\Phi(t_0,\chi_0)~~\geq~~0 \;,~~D\Phi(t_0,\chi_0)~~=~~0 & \text{ and } & D^2\Phi(t_0,\chi_0)~~\geq~~0\;.\qquad
\enq
Let $(\eps_{n},t_n,\chi_n,P_n,Q_n)_{n\geq1}\subset [0,1]\x [0,T]\x \L_{2}(\tilde \Omega^{\i},\tilde \Fc^{\i},\tilde \P^{\i};\R^d)\x \L_{2}(\tilde\Omega^{\i},\tilde\Fc^{\i},\tilde\P^{\i};$ $\R^d)\x S(\L_{2}(\tilde \Omega^{\i},\tilde \Fc^{\i},\tilde \P^{\i};\R^d))$ be a sequence converging to $(0, t_0,\chi_0,D\Phi(t_0,\chi_0),$ $D^2\Phi(t_0,$ $\chi_0))$ such that 
\beq\label{approx-H_*1}
\Hc_{\eps_{n}}(t_n,\chi_n,P_n,Q_n) & \to & \Hc_*(t_0,\chi_0,D\Phi(t_0,\chi_0),D^2\Phi(t_0,\chi_0)) \;.
\enq
It follows from \reff{cond-1-sub-sol} that
\beqs
&&\lim_{n\rightarrow+\infty}\Hc_{\eps_{n}}(t_n,\chi_n,P_n,Q_n)\\
& & \leq  \lim_{n\rightarrow+\infty}-{1\over 2} \inf_{u\in \L_{0}(\tilde \Omega^{\i},\tilde \Fc^{\i},\tilde \P^{\i};\Ur)}\tilde \E\Big[ Q_n (a_{t_n}(\chi_n,\tilde \P_{\chi_n},u)Z)a_{t_{n}}(\chi_n,\tilde \P_{\chi_n},u)Z \Big].
\enqs
Since $a$ is continuous  and bounded, it follows from the convergence of $Q_n$ to $D\Phi(t_0,\chi_0)$ that 
 \beqs
   \lim_{n\rightarrow+\infty}\inf_{u\in  \L_{0}(\tilde \Omega^{\i},\tilde \Fc^{\i},\tilde \P^{\i};\Ur)}\tilde \E\Big[ Q_n(a_{t_n}(\chi_n,\tilde \P_{\chi_n},u)Z)a_{t_{n}}(\chi_n,\tilde \P_{\chi_n},u)Z \Big] & = &\\
\inf_{u\in  \L_{0}(\tilde \Omega^{\i},\tilde \Fc^{\i},\tilde \P^{\i};\Ur)}\tilde \E\Big[  D^2\Phi(t_0,\chi_0)(a_{t_0}(\chi_0,\tilde \P_{\chi_0},u)Z)a_{t_{0}}(\chi_0,\tilde \P_{\chi_0},u)Z\Big]. 
& &  \enqs 
Combining the above leads to
 \beqs
&&\lim_{n\rightarrow+\infty}\Hc_{\eps_{n}}(t_n,\chi_n,P_n,Q_n) \\
&& \leq  -{1\over 2}\inf_{u\in \L_{0}(\tilde \Omega^{\i},\tilde \Fc^{\i},\tilde \P^{\i};\Ur)}  \E\Big[D^2\Phi(t_0,\chi_0)(a_{t_0}(\chi_0,\tilde \P_{\chi_0},u)Z)a_{t_{0}}(\chi_0,\tilde \P_{\chi_0},u)Z\Big]\;,
 \enqs
 so that \reff{cond-1-sub-sol} and \reff{approx-H_*1}  lead to 
  \beqs
-\partial_t\Phi(t_0,\chi_0)+ \Hc_*(t_0,\chi_0,D\Phi(t_0,\chi_0),D^2\Phi(t_0,\chi_0)) & \leq & 0\;.
 \enqs
 
\vspace{2mm}

\ni \textbf{2.} Suppose now that $V^*(t_0, \chi_0) = 1$.  
We argue by contradiction and suppose that
\beqs
-\partial_t \Phi(t_0, \chi_0) + \Hc_* \big(t_0, \chi_0, D \Phi(t_0, \chi_0), D^2\Phi(t_0, \chi_0)\big)  & =: & 4\eta~~>~~ 0\;. 
\enqs 
Since the left hand-side is finite and $\Nc_{0}\subset \Nc_{\eps}$ for $\eps\ge 0$, there exists an open neighborhood $\Oc$ of $(t_0, \chi_0, D \Phi(t_0, \chi_0))$ such that $\Nc_{0}\neq\emptyset$ on  $  \Oc$ and there exists $u_0\in \Nc_0(t_0,\chi_0,D\Phi(t_0,\chi_0))$ such that
\beqs
-\partial_t \Phi(t_0, \chi_0) - \Lc_{t_0}^{u_0} \big(t_0, \chi_0, D \Phi(t_0, \chi_0), D^2\Phi(t_0, \chi_0)\big)  & \geq & 2\eta\;.
\enqs
Then, \textbf{(H2)} implies that for any $\eps>0$ there exists  an open neighborhood $\Oc'$ of $(t_0, \chi_0, D \Phi(t_0, \chi_0))$ and 
   a measurable map $\hat u:[0,T]\x \R^{d} \x \R^{d}\x \tilde \Omega^{\i}$ $\to $ $\Ur$  such that:
\\
\noindent{\rm (i)} $\tilde \E_{B}[|\hat u_{t_{0}}(\chi_0,P_0,\xi)-u_0|]\leq \eps$\\
\noindent{\rm (ii)} There exists $C>0$ for which 
\begin{align*}
\tilde \E[|\hat u_{t}(\chi, P,\xi)-\hat u_{t}(\chi', P',\xi)|^{2}]\le C\tilde \E[|\chi-\chi'|^{2}+  |P-P'|^2] 
\end{align*}
for all $(t,\chi,P),(t,\chi',P')\in \Oc'$.\\
\noindent{\rm (iii)} $\hat u_{t}(\chi, P,\xi)\in \Nc_{0}(t,\chi,P)$ $\P^{\circ}-a.e.$, for all $(t,\chi,P)\in \Oc'$.

Define  
\beqs
\tilde \Phi(t,\chi) & = & \Phi(t,\chi)+|t-t_0|^2 + \tilde \E_{B}\big[|\chi-\chi_0|^2\big]^2\;,
\enqs
for  $(t,\chi)\in [0,T]\times\L_{2}(\tilde \Omega,\tilde \Fc,\tilde \P;\R^d)$. Then, 
\begin{align*}
(\partial_t\tilde \Phi,D\tilde \Phi,D^2\tilde \Phi)(t_0,\chi_0) =(  \partial_t\Phi,D  \Phi,D^2  \Phi)(t_0,\chi_0).
\end{align*}
The above combined with  \textbf{(H1)}-\textbf{(H1')} shows that 
we can find some $\eps>0$ such that 
\beq\label{local-estim-gen}
-\partial_t \tilde \Phi(t, \chi) - \Lc_{t}^{\hat u_{t}(\chi, D\tilde \Phi(t,\chi),\xi)}(\chi,D \tilde \Phi(t, \chi),D^{2} \tilde \Phi(t, \chi))   \geq \eta
\enq
for all $(t,\chi)\in B_\eps(t_0,\chi_0)$.

Let now $(t_n,\chi_n)_{n\ge 1}$ be a sequence of $[0,T]\times \L_{2}(\tilde \Omega^{\i},\tilde \Fc^{\i},\tilde \P^{\i};\R^d)$ such that
\beq\label{approx-seq-sub-sol}
\big(t_n,\chi_n,{V}(t_n,\chi_n)\big) & \to & \big(t_0,\chi_0,{V^*}(t_0,\chi_0)\big)\;,
\enq
and consider the solution $X^n$  of \reff{eq: SDEMF} starting from $\chi_n$ at $t_n$ and associated to the feedback control $\hat \nu^{n}:=\hat u_{\cdot}(X^n,D\tilde \Phi(.,X^n),\xi)$. The fact  that $X^n$ is well-defined  is guaranteed by {\rm (ii)} above, this is obtained by a straightforward extension of Proposition \ref{prop: existence unicite eds}. 
We then define the stopping times $\theta_n$ by
\beqs
\theta_n(\omega^{\circ}) & = & \inf\big\{s\geq t_n~:~(s,X^n_s(\omega^{\circ},.)\notin B_\eps(t_n,\chi_n)\big\} \;,\quad \omega^{\circ}\in \Omega^{\circ}\;.
\enqs
Letting 
\beqs
-\zeta   :=   \max_{\partial_pB_\eps(t_0,\chi_0)} (V^*-\tilde \Phi)~~<~~0  \;,
\enqs
we have
$(V-\Phi)(\theta_n,X^n_{\theta_n})  \leq  -\zeta$.

We then apply Proposition \ref{RK-nonlif-CR}, to  deduce from {\rm (iii)} and \reff{local-estim-gen} that 
$\tilde \Phi (\theta_n ,$ $X^n_{\theta_{n} } )  \leq   \tilde \Phi (t_n,\chi_n)$
which implies  $V (\theta_n ,X^n_{\theta_{n} })   \leq   \tilde \Phi (t_n,\chi_n)-\zeta$. Since $\tilde \Phi(t_{n},\chi_{n})\to 1$, we have $V (\theta_n ,X^n_{\theta_{n} }) <1$ for $n$ large enough, which contradicts Theorem \ref{Th-DDP}.
\ep
%

{We end this section with the derivation of the boundary condition at the terminal time $T$. 
To this end, let us define the function  $g=1-\mathds{1}_{\bar G}$ where  
\beqs
\bar G & = & \big\{\chi \in \L_{2}(\tilde \Omega^{\i},\tilde \Fc^{\i},\tilde \P^{\i};\R^d)~:~\tilde \P_\chi\in G\big\}.
\enqs
Note that $\bar G$ is a closed subset of  $\L_{2}(\tilde \Omega^{\i},\tilde \Fc^{\i},\tilde \P^{\i};\R^d)$ since $G$ is closed for $\Wc_2$. Hence, 
\beqs
g^*  =  1-\mathds{1}_{{\rm int}(\bar G)}\;,\;
g_*  =  1-\mathds{1}_{\bar G},
\enqs
 where $g^*$ and $g_*$ stand for the upper and lower semi-continuous envelopes of $g$ respectively.
}
\begin{Theorem}\label{THMTC}
Under \textbf{(H1)}, the function $V$ satisfies
\beqs
V^*(T,.)~~= g^* & \mbox{ and } & V_*(T,.)~~= g_* 
\enqs
on $\L_{2}(\tilde \Omega^{\i},\tilde \Fc^{\i},\tilde \P^{\i};\R^d)$.
\end{Theorem}
\begin{proof}
(i) We first prove that $V^*(T,.)=g^*$. Since $V(T,.)=g$, we have $V^*(T,.)\geq g^*$. For the reverse inequality, we argue by contradiction and suppose that   $1=V^*(T,\chi)> g^*(\chi)=0$ for some $\chi \in \L_{2}(\tilde \Omega^{\i},\tilde \Fc^{\i},\tilde \P^{\i};\R^d)$. 
Since $g^*(\chi)=0$, we know that $\chi\in {\rm int}(\bar G)$. 
Let $(t_n,\chi_n)_n$ be a  sequence such that
$(t_n,\chi_n,V(t_n,\chi_n)) \to (T,\chi,1)$.
Fix some $u_0\in \rm U$ and denote by $X^{t_n,\chi_n,u_0}$ the solution to \reff{eq: SDEMF} starting from $\chi_n$ at $t_n$ and controlled by the constant processes $\nu=u_0$. Then, $X_T^{t_n,\chi_n,u_0}\in \bar G ^c$,  after possibly considering a subsequence. Sending $n$ to $\infty$, we obtain that $\chi$ belongs to the closure of $\bar G^c$, which is a contradiction.

\noindent (ii) We now prove that  $V_*(T,.)=g_*$. Since $V(T,.)=g$ we have $V_*(T,.)\leq g_*$. Again the reserve inequality is proved by contradiction. Suppose that $0=V_*(T,\chi)< g_*(\chi)=1$  for some $\chi \in \L_{2}(\tilde \Omega^{\i},\tilde \Fc^{\i},\tilde \P^{\i};\R^d)$. Since $g_*=g$,   we know that $\chi\in \bar G^{c}$.
Let $(t_n,\chi_n)_n$ be a  sequence such that
$
(t_n,\chi_n,V(t_n,\chi_n)) \to (T,\chi,0)$.
Then, up to taking a subsequence, there exists $\nu^n\in \Uc$ such that $X_T^{t_n,\chi_n,\nu_n}\in \bar G$. Since  $a$ and $b$ are continuous  bounded and $\bar G$ is closed in $\L_{2}(\tilde \Omega^{\i},\tilde \Fc^{\i},\tilde \P^{\i};\R^d)$, we deduce that $\chi\in \bar G$ by sending $n$ to $\infty$, which is a contradiction.
\end{proof}

\begin{Remark} Note that the terminal condition in Theorem \ref{THMTC} is discontinuous, which prevents us from proving uniqueness of a solution to our PDE. This point will be further discussed in Section \ref{sec: other formulation} below.
\end{Remark}

\section{Additional remark{s}}\label{sec: Uccirc}

\subsection{{On the formulation}}\label{sec: other formulation}

The formulation considered in this paper naturally leads to a PDE characterization with a discontinuous terminal condition (upper- and lower-semi-continuous enveloppes of $1-\1_{G}$). Even for PDEs stated on a subset of $\R^{d}$ this is problematic from a numerical point of view, in particular because comparison does not hold. In some cases, an alternative formulation can be used in order to retrieve a regular terminal condition and open the door to the study of comparison and possibly of numerical methods  by using already existing results on PDE's on Hilbert spaces,  see e.g.~\cite{fabbri2017stochastic}.\footnote{Note that, even for general stochastic target problems set on $\R^{d}$, no general comparison theorem has been established so far. This is done on a case by case basis, and we therefore do not enter into this issue in the abstract setting of this paper, but rather leave this to the future study of particular situations.} Let us discuss this in the context of Example \ref{example}.

We consider the same problem as in  Example \ref{example} but now take the cost  induced by the fertilizing effort of each particle into account.  Its dynamics is of the form:
\begin{align*}
C^{t,\nu}&=\int_{t}^{\cdot} b^{C}(\nu_{s})ds,
\end{align*}
in which  $b^{C}$ is non-negative. 
The initial budget of the farmer at $t$ is $y\in \R$, and we set $Y^{t,y,\nu}:=y-C^{t,\nu}_{\cdot}$, so that $\E_{B}[Y^{t,y,\nu}]$ denotes the remaining running budget: initial budget minus integral with respect to the Lebesgues measure of the costs associated to each particle.
Letting $\hat X^{t,\chi,\nu}:=(X^{t,\chi_{X},\nu},Y^{t,y,\nu})$, with $\chi=(\chi_{X},y)$, we retrieve the dynamics \eqref{eq: SDEMF} for $\hat X^{t,\chi,\nu}$. The aim of the farmer is to find the minimal initial budget $y$ and a control $\nu$ such that $\P^{B}_{X^{t,\chi,\nu}_{T}}\in G_{X}$ and $\E_{B}[Y^{t,y,\nu}_{T}]\ge 0$ $\P$-a.s.~for some closed subset $G_{X}$ of the collection of probability measures with second order moment. Otherwise stated, he aims at computing  at $t$ how much money should be put aside to cover with certainty\footnote{One could relax the constraint by just asking for $\P[\E_{B}[Y^{t,y,\nu}_{T}]\ge 0]\ge m$ for some $m\in (0,1)$, see \cite{BET10}.} the costs of driving the field in a given set of acceptable states at time $T$.  

In this context, let us define\footnote{The state space being increased to $\R^{d+1}$.}, for $t\in[0,T]$ and $\mu_{X}\in \Pc_2$,
$$
v(t,\mu_{X}):=\inf\{y\in \R: (\mu_{X},\delta_{y})\in  \Vc(t)\}
$$
where $\delta_{y}$ is the Dirac mass at $y$ and $\Vc$ is defined with respect to $G=G_{X}\times G_{Y}$ for $G_{Y}$ defined as the collection of probability measures with support on $\R$,  with finite second order moments and non-negative first order moment. The dynamic programming principle of Theorem \ref{Th-DDP} reads as follows :

{\rm (GDP1)} If $y>v(t,\mu_{X})$ then there exists $\nu \in \Uc$ and $(\chi_{X},y)\in \Xb^{2}_{t}$  such that $\E_{B}[Y^{t,y,\nu}_{\theta}] \ge v(t,\P^{B}_{X^{t,\chi_{X},\nu}_{\theta}})$  and $\P^{B}_{\chi_{X}}=\mu_{X}$  $\P$-a.s. 

{\rm (GDP2)}  If there exists  $\nu \in \Uc$ and $(\chi_{X},y)\in \Xb^{2}_{t}$  such that $\E_{B}[Y^{t,y,\nu}_{\theta}] > v(t,\P^{B}_{X^{t,\chi_{X},\nu}_{\theta}})$ and $\P^{B}_{\chi_{X}}=\mu_{X}$  $\P$-a.s., then   $y\ge v(t,\mu_{X})$. 

Indeed, $y>v(t,\mu_{X})$ implies that $(\mu_{X},\delta_{y})\in \Vc(t)$, which by Theorem \ref{Th-DDP} induces that $(\P^{B}_{X^{t,\chi_{X},\nu}_{\theta}},\P^{B}_{Y^{t,y,\nu}_{\theta}})\in \Vc(\theta)$, for some  $\nu \in \Uc$ and $(\chi_{X},y)\in \Xb^{2}_{t}$ such that $\P^{B}_{\chi_{X}}=\mu_{X}$. Since   $\E_{B}[Y^{t,\chi_{Y},\nu}_{T}]\ge 0$ $\P$-a.s.~for some   $\chi_{Y}\in \L_{2}(\Omega^{\i},\Fc^{\i}_{\theta},\P;\R)$ is equivalent to saying that $\E_{B}[Y^{t,y,\nu}_{T}]\ge 0$ $\P$-a.s.~for $y:=\E^{B}[\chi_{Y}]$, this implies that $(\P^{B}_{X^{t,\chi_{X},\nu}_{\theta}},\delta_{\E_{B}[Y^{t,y,\nu}_{\theta}]})\in \Vc(\theta)$. Conversely, $\E_{B}[Y^{t,y,\nu}_{\theta}] > v(t,\P^{B}_{X^{t,\chi_{X},\nu}_{\theta}})$ and $\P^{B}_{\chi_{X}}=\mu_{X}$ $\P$-a.s.~implies that $(\P^{B}_{X^{t,\chi_{X},\nu}_{\theta}},\P^{B}_{Y^{t,y,\nu}_{\theta}})\in \Vc(\theta)$. 
 
 From this version of the geometric dynamic programming principle, it is not difficult to adapt the arguments of Section \ref{subsec: visco cara}, see e.g.~\cite{BET10,ST02a}, to derive that the lift $V$ of $v$ is such that $V_{*}$ and $V^{*}$ (if finite, e.g.~because $b^{C}$ is bounded) are respectively viscosity super- and subsolutions of 
   \eqref{eq: EDPV}, on the corresponding space (now associated to the $X$ component above only), with terminal conditions $V_{*}(T,\cdot)\ge 0\ge V^{*}(T,\cdot)$, up to mild regularity conditions on the coefficients. 

\subsection{On the choice of controls}
{In the above sections, the collection $\Uc$ of controls permits to take into account the exact value of the initial random variable $\chi$, it is $\F$-progressively measurable. If we think in terms of controlling a population of particles whose initial distribution is the law of $\chi$, this means that we allow each of the particles to have its own control. One can also consider the case where the control belongs to the subclass $\Uc^{\circ}$ of controls in $\Uc$ that are only $\bar \F^{\circ}$-progressively measurable. This would mean that the control of each particle does not depend on its  position but only of the conditional law of the whole population of particles given $B$. 
\\
This can be treated in a similar way as the case we considered above. In particular, the result of Proposition \ref{prop: equivWS} becomes trivial, see Proposition \ref{prop: equiv loi jointe}. In \reff{eq: concat control}, the control $\nu$ will be $\bar \F^{\circ}$-progressively measurable and the map $\vartheta$ will take values in $\Uc^{\circ}$, so that $\bar\nu$ will actually be $\bar \F^{\circ}$-progressively measurable since the argument  $X^{t,\chi,\nu}_{\theta}(\omega^{\circ},\cdot)$ only enters as a random variable (not as the value of the random variable). As for the first part of the proof of Theorem \ref{Th-DDP}, the construction will just be simpler. Then,  Theorem \ref{Th-DDP} actually holds for the class  $\Uc^{\circ}$ as well. As for the PDE characterization of Theorem \ref{THMBT}, we only have to replace  $\Nc_{\eps}(t,\chi,P)$ with
$ \{ u\in \Ur~:~|\E_{B}[a_t(\chi,\P_\chi,u)P]|\le \eps  \}$, which changes the definition of $\Hc^{*}$ and $\Hc_{*}$ accordingly. Up to this modification, the proof is the same. }

 \bibliographystyle{plain}
\bibliography{BibBDK}

\end{document}